\documentclass{article}

\usepackage{PRIMEarxiv}

\usepackage[utf8]{inputenc} 
\usepackage[T1]{fontenc}    
\usepackage{hyperref}       
\usepackage{url}            
\usepackage{booktabs}       
\usepackage{amsfonts}       
\usepackage{nicefrac}       
\usepackage{microtype}      

\RequirePackage{amsthm,amsmath,amsfonts,amssymb}
\RequirePackage[numbers]{natbib}

\usepackage{relsize}
\usepackage{dsfont}

\usepackage{fancyhdr}       
\usepackage{graphicx}       
\graphicspath{{media/}}     

\newtheorem{theorem}{Theorem}[section]
\newtheorem{remark}[theorem]{Remark}

\theoremstyle{plain} 
\newtheorem{thm}{Theorem}[section]

\newtheorem{defn}[thm]{Definition}
\newtheorem{ass}[thm]{Assumption} 

\numberwithin{equation}{section}

\title{Feynman-Kac formula for BSDEs with jumps and time delayed generators associated to path-dependent nonlinear Kolmogorov equations
}

\author{
  Luca Persio \\
  Department of Computer Science\\
  University of Verona \\
  Strada le Grazie 15, Verona, 37134, Italy\\
  \texttt{luca.dipersio@univr.it} \\
  \And
   Matteo Garbelli  \\
  Department of Mathematics \\
  University of Trento \\
  Via Sommarive 14, Povo (Trento), 38123, Italy\\
  \texttt{matteo.garbelli@unitn.it} \\
   \And
  Adrian Z\u{a}linescu \\
  Faculty of Computer Science \\
  Alexandru Ioan Cuza University \\
  General Berhtelot Street 16,
Iasi, 700483, Romania\\
  \texttt{adrian.zalinescu@info.uaic.ro} \\
}

\begin{document}
\maketitle

\begin{abstract}
We consider a system of Forward Backward Stochastic Differential Equations (FBSDEs), with time delayed generator and driven by Lévy-type noise.
We establish a non-linear Feynman–Kac representation formula associating the solution given by the FBSDEs-system to the solution of a path dependent nonlinear Kolmogorov equation with both delay and jumps.
Obtained results are then applied to study a generalization of the so-called {\it Large Investor Problem} where the stock price evolves according to a jump-diffusion dynamic.
\end{abstract}


\section{Introduction}

Stochastic delay differential equations are derived as a natural generalisation of stochastic ordinary differential equations by allowing coefficients' evolution to depend not only on the present state, but also on past values.

In this article we analyze a stochastic process described by a system of Forward-Backward Stochastic Differential Equations (FBSDEs). The forward path-dependent equation is driven by a Lévy process, while the backward one presents a path-dependent behaviour with dependence on small delay 
$\delta$.

We establish a non-linear Feynman–Kac representation formula associating the solution of latter FBSDE system to the one of a path dependent nonlinear Kolmogorov equation with delay and jumps. In particular, we prove that the stochastic system allows to uniquely construct a solution to the parabolic Partial Differential Equation(PDE), in the spirit of the Pardoux-Peng approach, see \cite{PardouxPeng1, PardouxPeng2}, i.e.:
\begin{small}
\begin{equation}
\label{Kolmogorov}
\begin{cases}
- \partial_t u \big( t , \phi \big) - \mathcal{L} u \big(t, \phi \big) - f \big( t, \phi, u (t, \phi) , \partial_x u (t, \phi) \sigma (t, \phi),(  u ( \cdot , \phi ))_t , \mathcal{J}  u (t, \phi) \big) = 0 \, , \smallskip \\
u (T, \phi) = h (\phi) \, , \hspace{1.2cm} (t, \phi) \in [0,T] \times \Lambda \, , 
\end{cases}
\end{equation}
\end{small}
being $T< \infty$ a fixed time horizon and $\Lambda$ being $\mathcal{D} ([0,T]; \mathbb{R}^d)$, i.e. the space of càdlàg $\mathbb{R}^d$-valued functions defined on the interval $[0,T]$. The second order differential operator $\mathcal{L}$, associated to the diffusion, is defined by
\begin{equation}
    \label{operatorL}
\mathcal{L} u \big(t, \phi \big) := \dfrac{1}{2} \textsl{Tr } [\sigma (t, \phi) \sigma^{*} (t, \phi) \partial_{xx}^2  u (t, \phi) ] + \langle b (t, \phi) , \partial_x u (t, \phi) \rangle \, , 
\end{equation}
with $ b: [0,T] \times \Lambda  \rightarrow \mathbb{R}^d$ and $\sigma:  [0,T] \times \Lambda  \rightarrow \mathbb{R}^{d \times l}$ being two non-anticipative functionals while $\mathcal{J}$ defines the integro-differential operator associated to the jump behaviour
\begin{equation}
\label{jump}
\mathcal{J} u (t, \phi) := \int_{\mathbb{R} \setminus \{0\}} \big[ u \big (t, \phi  + \gamma (t, \phi, z) \big) - u (t,\phi)\big] \lambda(z) \nu(dz) \,  , 
\end{equation}
where $\gamma :  [0,T] \times \Lambda
\times \left\{ \mathbb{R} \setminus \{0\} \right\} \rightarrow \mathbb{R}^d $ is a continuous, non anticipative functional modelling the jump component while  $\lambda$ models the intensity of the jumps and $\nu$ represents the associated Lévy measure.

Also, for a fixed delay $\delta >0$ we set 
\begin{equation}
    \label{delayed}
 (  u ( \cdot , \phi ))_t := (  u ( t + \theta , \phi ))_{\theta \in [-\delta,0]} \, , 
 \end{equation}

We emphasize that the study of a path-dependent Kolmogorov equation whose generator $f$ depends on both a delayed term $(  u ( \cdot , \phi ))_t$ and on a jump operator $\mathcal{J} u (t, \phi) $,  represents the main novelty we provide in this paper.

Under appropriate assumptions on the coefficients, the deterministic non-anticipative functional $ u : [0,T] \times \Lambda \rightarrow \mathbb{R}$ given by the representation formula
\begin{equation}
\label{feynman}
 u (t, \phi) := Y^{t, \phi} (t) \, , 
 \end{equation}
is a mild solution of the Kolmorogov Equation \eqref{Kolmogorov}.

Concerning the stochastic process $Y^{t, \phi} (t)$ in Eq. \eqref{feynman}, provided conditions in Assumptions \ref{ass1} - \ref{ass2}, are fulfilled, we know that the quadruple 
$$\big(X^{t, \phi},Y^{t, \phi},Z^{t, \phi},U^{t, \phi} \big)_{s \in [t, T]} \, , $$
is the unique solution of the system of FBSDEs on $[t,T]$ given by

\begin{equation}
\label{fbsde}
\begin{cases}
 \displaystyle X^{t, \phi} (s) = \phi (t)  +  \int_t^s b (r, X^{t,\phi} ) dr + \int_t^s \sigma (r, X^{t, \phi} ) d W (r) \smallskip \\
 \displaystyle \qquad \quad
 + \int_t^s  \int_{\mathbb{R} \setminus \{ 0 \}} \gamma (r, X^{t,\phi},z) \tilde{N} (dr, dz) \smallskip \\
 \displaystyle Y^{t, \phi} (s) = h ( X^{T, \phi} ) +  \int_s^T f \big( r, X^{t, \phi}, Y^{t, \phi} (r), Z^{t, \phi} (r), \tilde{U}^{t, \phi} (r), Y_r^{t,\phi} \big) dr \smallskip \\ 
 \displaystyle \qquad \quad -  \int_s^T  Z^{t, \phi} (r) dW(r) -  \int_s^T  \int_{\mathbb{R} \setminus \{ 0 \}} U^{t,\phi} (r,  z) \tilde{N} (dr, dz) \smallskip \, , 
\end{cases}
\end{equation}

where $W$ stands for a $l$-dimensional standard Brownian motion.
We assume a delay $\delta \in \mathbb{R}^+$. Consequently, the notation $Y_r^{t, \phi}$, appearing in the generator $f$ of the backward dynamic in the system \eqref{fbsde}, stands for the delayed path of the process $Y^{t, \phi}$ restricted to $[r -\delta,r]$, namely
\begin{equation}
\label{delayed_term}
Y_r^{t, \phi} := \big(Y^{t, \phi} (r + \theta)\big)_{\theta \in [-\delta, 0]} \, .
\end{equation}



In Eq. \eqref{fbsde}, $\tilde{N}$ models a compensated Poisson random measure, independent from $W$, with associated Lévy measure $\nu$. This stochastic term appears also in the definition of the integral term $\tilde{U} : [0,T] \rightarrow \mathbb{R}$ related to the jump process by
\begin{equation}
\label{Utilde}
\tilde{U}^{t, \phi} (r) = \mathlarger{\int}_{\mathbb{R} \setminus \{ 0 \}}  U^{t, \phi} (r,z) \lambda(z) \nu (dz),
\end{equation}
we need to introduce to express the solution of \eqref{Kolmogorov} via the FBSDE system \eqref{fbsde}.

The connection between probability theory and PDEs is a widely analysed subject, the well-known Feynman-Kac formula (\cite{Kac}) being one of its main turning points stating  that solutions for a large class of second order PDEs of both elliptic and parabolic type, can be expressed in terms of expectation of a diffusion process. 
Latter result has been then generalised by Pardoux and Peng in \cite{PardouxPeng1, PardouxPeng2} to show the connection between Backward Stochastic Differential Equations (BSDEs) and a system of semi-linear PDEs, then proving the non linear Feynman-Kac formula within the Markovian setting.
Concerning the non-Markovian scenario, 
we know from \cite{peng1, PengWang} that 
a non linear Feynman-Kac formula can be still established, associating a path-dependent PDE to a non-Markovian BSDE.
More recently, the introduction of horizontal and vertical derivatives of non-anticipative functionals on path
spaces by Dupire \cite{dupire}, Cont \cite{Cont} and Fournié \cite{fournie}, facilitated the formulation of a new class of path dependent PDEs and the introduction of the so-called {\it viscosity solution} concept, see \cite{touzi1}, \cite{touzi2}, \cite{PengWang}, for more details. For a complete overview about stochastic calculus with delay we refer to \cite{sea1}, \cite{sea2} and \cite{sea3}.

A connection between time-delayed BSDEs and path-dependent PDEs has been proved by an {\it infinite dimensional lifting} approach in  \cite{FlandoliZanco} and \cite{MOTZ}.  In \cite{buck}, the authors considers a BSDE driven by a Brownian motion and a Poisson random measure that provides a viscosity solution of a system of parabolic integral-partial differential equations. In \cite{CDMZ}, the existence of a viscosity solution to a path dependent nonlinear Kolmogorov
equation (without jumps) and the corresponding nonlinear
Feynman–Kac representation has been proved.

 In this paper we deal with the notion of {\it mild solution} which can be seen as intermediate notion for solution of a PDE lying in between the notions of classical and viscosity solutions. In \cite{nonlinear}, the authors provide the definition of mild solution for non linear Kolmogorov equations along with its link with a specific stochastic process. The latter has been also proved for Semilinear Parabolic Equations in \cite{generalized}, where the definition of the {\it generalized directional gradient} is firstly introduced. The concept of  mild solution together with the generalized directional gradient to handle path-dependent Kolmogorov equation with jumps and delay has been widely analyzed in the functional formulation, see,  e.g., \cite{Oliva}. Moreover, a discrete-time approximation for  solutions of a system of decoupled FBSDEs with jumps have been proved in \cite{be} by means of  Malliavin calculus tools.

Concerning the theory of BSDE with a dependence on a delay, in \cite{DelongImkeller}, the authors proved the existence of a solution for a BSDE with a time-delayed generator that depends on the past values of the solution. In particular, both existence and uniqueness are proved assuming a sufficiently small time horizon $T$ or a sufficiently small Lipschitz constant for the generator. Let us underline that the latter has an equivalent within our setting, as we state in Remark \ref{small}. Moreover, in \cite{Delong1, DelongImkeller2} the authors defined a path-dependent BSDE with time delayed generators driven by
Brownian motions and Poisson random
measures, with coefficients depending on the whole solution's path.

In \cite{Banos}, following a different approach, namely considering systems with memory and jumps, the authors provide a characterization of a strong solution for a delayed SDE with jumps,
considering both $L^p$-type space and càdlàg processes to provide a non-linear Feynman–Kac representation theorem.


\bigskip


The present paper is structured as follows: we start providing notations and problem setting in Section \ref{sec_Poisson}, according to the theoretical framework developed by \cite{Delong1} and \cite{DelongImkeller}; in Section \ref{sec3} we study the well-posedness of the path-dependent BSDE mentioned appearing in the Markovian FBSDEs stystem \eqref{fbsde} following the approach in \cite{CDMZ} by additionally considering jumps; in Section \ref{sec4} we provide a Feynman-Kac formula relating the BSDE to 
the Kolmogorov Equation defined in \eqref{Kolmogorov} to then generalise results in \cite{Oliva} by considering a dependence in the generator $f$ of the backward dynamic on a delayed $L^2$ term, namely $Y_r^{t,\phi}$, for a small delay $\delta$; 
in Section \ref{sec5} we derive the existence of a mild solution for the Kolmogorov Equation within the setting developed in \cite{generalized}; finally, in Section \ref{SEC:FA}, we provide an application based on the analyzed theoretical setting, i.e. a version of the {\it Large Investor Problem} characterised by a jump-diffusion dynamic.

\section{Notation and Problem formulation}
\label{sec_Poisson}
On a probability space $\big( \Omega, \mathcal{F}, \mathbb{P} \big)$, we consider a standard $l$-dimensional Brownian motion $W$ and a homogeneous Poisson random measure $N$ on $\mathbb{R}^{+}\times(\mathbb{R}\setminus \{0\})$, independent from $W$, with intensity $\nu$. 
With the notation $\mathbb{R}_0 := \mathbb{R}\setminus \{0\}$, we also define the compensated Poisson random measure $\tilde{N}$ defined on $\mathbb{R}^+ \times\mathbb{R}_0$ by
\begin{equation}
\label{poisson}
\tilde{N} (dt, dz) := N(dt, dz) - \nu (dz)dt \, .
\end{equation}
For the sake of completeness, we recall that
the term $ \nu (dz)dt$  is known as the {\it compensator} of the random measure $N$.
Moreover, we  assume that the Lèvy measure $\nu$ satisfies
\begin{equation}
    \label{levy}
 \int_{\mathbb{R}_0} |z|^2\nu (dz) < \infty\,,
 \end{equation}
 being a standard assumption when dealing  with financial applications. The reader could see, e.g., in \cite{Bass} further details about the stochastic integration with the presence of jumps.

\subsection{The forward-backward delayed system}

\label{fbsystem}

In this section we introduce the delayed forward-backward system, assuming path dependent coefficients for the forward and the backward components and  a dependence on a small delay into the generator $f$ and
the  presence of jumps modeled via a compound Poisson measure. Furthermore, the equation is formulated on a general initial time $t$ and initial values. Thus, we need to equip the backward equation with a suitable condition in $ [0,t)$, as we introduced in Equation \eqref{supplementary}.

On previously defined probability space, we consider a filtration $ \mathbb{F}^t = \{ \mathcal{F}_s^t \}_{s \in [0,T]}$, which is nothing but the one jointly generated by $W (s \wedge T) - W (t)$ and $N(s \wedge T,\cdot) - N(t,\cdot)$, augmented by all $\mathbb{P}$-null sets. We emphasize that $ \mathbb{F}^t $ depends explicitly on $t$, namely the arbitrary initial time in $[0,T]$ for the dynamic in Eq. \eqref{fbsde}.

Furthermore,the components of the solution of the backward dynamic are defined in the following Banach spaces: 
 
\begin{itemize}

\item $\mathbb{S}_t^2 (\mathbb{R}) $ denotes the space of (equivalence class of) $\mathbb{F}^t$-adapted, product measurable càdlàg processes $Y: \Omega
\times [0,T] \rightarrow \mathbb{R}$ satisfying 
$$\mathbb{E} \Bigg[ \sup_{t \in [0,T]}|Y(t)|^2  \Bigg] < \infty \, ; $$

 \item  $\mathbb{H}^2_t (\mathbb{R}^{l}) $ denotes the space of (equivalence class of) $\mathbb{F}^t$-predictable processes $Z: \Omega
\times [0,T] \rightarrow \mathbb{R}^{ l}$ satisfying $$ \mathbb{E} \Bigg[ \int_0^T |Z(t)|^2 dt \Bigg] <\infty \, ; $$

\item $\mathbb{H}_{t,\nu}^2 (\mathbb{R}) $ denote the space of (equivalence class of) $\mathbb{F}^t$-predictable processes $U :\Omega \times [0,T] \times \mathbb{R}_0  \rightarrow \mathbb{R}$ satisfying
 $$ \mathbb{E} \Bigg[ \int_0^T \int_{\mathbb{R}_0} | U(t, z)|^2  \nu(dz) dt \Bigg]   < \infty \, .$$


\end{itemize}

The spaces $\mathbb{S}_t^2 (\mathbb{R}) $, $\mathbb{H}^2_t (\mathbb{R}^{l})$ and $\mathbb{H}_{t,\nu}^2 (\mathbb{R}) $ are endowed with the following norms:
$$|| Y ||^2_{\mathbb{S}_t^2 (\mathbb{R}) } = \mathbb{E} \Bigg[ \sup_{t \in [0,T]}  e^{\beta t}|Y(t)|^2  \Bigg] \, , $$
$$ || Z ||^2_{\mathbb{H}^2_t (\mathbb{R}^{ l})} = \mathbb{E} \Bigg[ \int_0^T e^{\beta t} |Z(t)|^2 dt \Bigg] \, , $$

and
$$ || U ||_{ \mathbb{H}_{t,\nu}^2 (\mathbb{R}) }^2 = \mathbb{E} \Bigg[ \int_0^T \int_{\mathbb{R}_0} e^{\beta t} | U(t, z)|^2  \nu(dz) dt  \Bigg] \, , $$

for some $\beta >0$, to be precised later.

The main goal is to find a family of stochastic processes $\Big( X^{t, \phi}, Y^{t, \phi} , Z^{t, \phi},  U^{t, \phi} \Big)_{t, \phi \in [0,T] \times \Lambda}$ adapted to $\mathbb{F}^t$ such that the following decoupled forward-backward system holds a.s.

\begin{small}
\begin{equation}
\label{fbsde1}
\begin{cases}

\displaystyle  X^{t, \phi} (s) = \phi (t)  +  \int_t^s b (r, X^{t,\phi} ) dr +  \int_t^s \sigma (r, X^{t, \phi} ) d W (r) + \smallskip \\
\qquad \qquad
\displaystyle
+  \int_t^s  \int_{\mathbb{R}_0} \gamma (r, X^{t,\phi},z) \tilde{N} (dr, dz) \, , \hspace{1.5cm}  s \in [t,T] \smallskip \\

X^{t,\phi} (s) = \phi (s) \, ,  \hspace{5cm}   \qquad s \ \in [0,t] \smallskip\\

\displaystyle
Y^{t, \phi} (s) = h (X^{t, \phi}) +  \int_s^T f \left( r, X^{t, \phi}, Y^{t, \phi} (r), Z^{t, \phi} (r),  \int_{\mathbb{R}_0} U^{t, \phi} (r,z) \lambda(z) \nu(dz), Y_r^{t,\phi} \right) dr  \smallskip \\ 
\qquad \qquad
\displaystyle
 -  \int_s^T  Z^{t, \phi} (r) dW(r) -  \int_s^T  \int_{\mathbb{R}_0} U^{t, \phi} (r, z) \tilde{N} (dr, dz) \,  , \hspace{.3cm} \qquad  s \in [t,T] \smallskip \\ 
 
Y^{t,\phi} (s) = Y^{s,\phi} (s) \, , \hspace{.3cm} Z^{t,\phi} (s) = U^{t,\phi} (s,z) = 0 \, , \hspace{1.2cm}  s \in [0,t] \, . \smallskip
\end{cases}
\end{equation}
\end{small}


It is worth to mention  that, differently from \cite{Delong1}, we work in a Markovian setting,  enforcing an initial condition over all the interval $[0,t]$. More precisely, for both forward and backward equations, the values of the solution $X^{t, \phi}$ need to be known in the time interval $[0,t]$.
Analogously, regarding the backward component, the values of $Y^{t, \phi}$, $Z^{t, \phi}$ and  $U^{t, \phi}$ need also to be prescribed for $s \in [0,t]$.

\begin{remark}
\label{delta}
The $\delta$-delayed feature concerns only $Y$, but we emphasize that it is possible to generalize this result to treat the case where both $Z$ and $U$ depend on their past values for a fixed delay $\delta$.

For the sake of simplicity, we consider the case with $Y_r$, hence limiting ourselves to just one, the process $Y$, $L^2$ delayed term. As a consequence, the latter implies that we will have to 
consider a larger functional space to properly define the contraction that is an essential step to prove the fix point argument in Th. \ref{theorem1}.

\end{remark}

\subsubsection{The forward path-dependent SDE with jumps}

We first study the forward component of $X$ appearing in the system \eqref{fbsde1}. It is defined according to the following equation:
\begin{equation}
\begin{cases}
\label{forward}
 \displaystyle X^{t, \phi} (s) = \phi (t)  +  \int_t^s b (r, X^{t,\phi} ) dr +  \int_t^s \sigma (r, X^{t, \phi} ) d W (r) + \smallskip \\
\qquad \displaystyle + \int_t^s  \int_{\mathbb{R}_0} \gamma (r, X^{t,\phi},z) \tilde{N} (dr, dz) \, , \hspace{1cm}  s \in [t,T] \smallskip  \\
X^{t,\phi} (s) = \phi (s) \, ,  \hspace{4.7cm}   s \ \in [0,t]  \, . \smallskip
\end{cases}
\end{equation}

 The existence and the pathwise uniqueness for a solution of \eqref{forward}  is a classical result, proved, for instance, in \cite{Sk}, exploiting a
 Picard iteration approach.
 We assume the following assumptions to hold:

\begin{ass}
\label{ass1}
Let us consider three non anticipative functions $b: [0,T] \times \Lambda \rightarrow \mathbb{R}^d$, $\sigma: [0,T] \times \Lambda \rightarrow \mathbb{R}^{d \times l}$ and $\gamma: [0,T] \times \Lambda \times \mathbb{R} \rightarrow \mathbb{R}^{d}$ such that

\begin{itemize}
\item[($A_1$)] $b$, $\sigma$ and $\gamma$ are continuous;
\item[($A_2$)] there exists $\ell > 0$ such that 
$$  | b(t, \phi) - b (t, \phi') | +  | \sigma(t, \phi) - \sigma (t, \phi') | +  ||\gamma (t, \phi, \cdot)- \gamma(t, \phi', \cdot)||_{L^2(\mathbb{R}_0,\nu)}  \,  \leq  \, \ell || \phi  - \phi'||_{L^\infty[0,T]} $$
for any $t \in [0,T]$, $\phi, \phi' \in \Lambda$;

\item[($A_3$)] the following bound $$\int_{\mathbb{R}_0} \sup_\phi |\gamma (r, \phi, z)|^2 \nu(dz)  < \infty \, ,$$ holds.

\end{itemize}
\end{ass}

If ($A_1$), ($A_2$), ($A_3$) hold, then there exists a solution to \eqref{forward} and this solution is pathwise unique, see, e.g., \cite{buck} for a detailed proof.


%

\subsubsection{The backward delayed path-dependent SDE with jumps}
\label{statespace}

We now focus on the BSDE appearing in the  system \eqref{fbsde1}, namely

\begin{small}
\begin{equation}
\label{backward}
\begin{cases}
\displaystyle Y^{t, \phi} (s) = h   (X^{t, \phi}) + \int_s^T f \left( r,  X^{t, \phi}, Y^{t, \phi} (r), Z^{t, \phi} (r), \int_{\mathbb{R}_0} U^{t, \phi} (r,z) \lambda(z) \nu(dz), Y_r^{t,\phi} \right) dr \medskip  \\ 
\displaystyle  \qquad  -  \int_s^T  Z^{t, \phi} (r) dW(r) - \int_s^T  \int_{\mathbb{R}_0} U^{t, \phi} (r, z) \tilde{N} (dr, dz) \,  , \hspace{.8cm}  s \in [t,T] \medskip \\ 
\displaystyle  Y^{t,\phi} (s) = Y^{s,\phi} (s) \, , \hspace{.8cm} Z^{t,\phi} (s) = U^{t,\phi} (s,z) = 0 \, , \hspace{1.9cm}  s \in [0,t] \, . \smallskip
\end{cases}
\end{equation}
\end{small}
for a finite time horizon $T < \infty$ and $\phi \in \Lambda := D \big( [0,T]; \mathbb{R}^d \big)$. The path-dependent process $X^{t, \phi}$ represents the solution of the forward SDE with jumps  of Eq. \eqref{forward}, while $\tilde{N}$ models the compensated Poisson random measure described in Eq. \eqref{poisson} and $W$ is a  $l$-dimensional Brownian motion.

We recall that, when we fix the delay term $\delta$, the notation $Y_r^{t,\phi}$ stands for the path of the process restricted to $[r- \delta, r]$, according to Eq.\eqref{delayed_term}. Notice that the terminal condition enforced by $h$ depends on the solution of the forward SDE \eqref{forward} as well as the solution $(Y, Z, U)$ of the backward component considered in the  time interval $[t,T]$.

Differently from the framework studied by Delong in \cite{Delong1}, we consider a general initial time $s \in [0,t)$. As highlighted in \cite{CDMZ}, the Feynman Kac formula would fail with standard prolongation. 

Thus, an additional initial condition has to be satisfied over the interval $[0,t]$, 
given by
\begin{equation}
\label{supplementary}
Y^{t, \phi} (s) = Y^{s, \phi} (s) \, , \hspace{.6cm} s \in [0,t) \, .
\end{equation}

We remark that the supplementary initial condition stated in Eq. \eqref{supplementary} represents one of the main difference between Theorem \ref{theorem1} and Theorem 14.1.1 in \cite{Delong1}.

\section{The Well Posedness of the BSDE}
\label{sec3}

Concerning the delayed backward SDE \eqref{backward}, we will assume the
following to hold.
\begin{ass}
\label{ass2}

Let $f:[0,T]\times{\Lambda}\times\mathbb{R} \times \mathbb{R}^{l} \times \mathbb{R} \times L^{2}\left( [-\delta,0];\mathbb{R} \right)  
\rightarrow \mathbb{%
R}$ and $h: [0,T] \times \Lambda  \rightarrow\mathbb{R}$ such that the
following holds:

\begin{itemize}

\item[($A_4$)]  There exist $L,K,M>0$, $p\geq 1$ and a
probability measure $\alpha $ on $\left( [-\delta ,0],\mathcal{B}\left( %
\left[ -\delta ,0\right] \right) \right) $ such that, for any $t\in \lbrack
0,T]$, $\phi \in {\Lambda }$, $\left( y,z, u\right) ,(y^{\prime
},z^{\prime }, u^{\prime})\in \mathbb{R} \times \mathbb{R}^{l} \times \mathbb{R}$ and $%
\hat{y},\hat{y}^{\prime }\in L^{2}\left( [-\delta ,0];\mathbb{R} \right) $, we have%
\begin{equation*}
\begin{array}{rl}
\left( i\right) & \displaystyle\phi \mapsto f\left(t,\phi ,y,z, u, \hat{y}\right)%
\text{ is continuous,}\smallskip \\ 
\left( ii\right) & \displaystyle|f(t,\phi ,y,z,u,\hat{y})-f(t,\phi
,y^{\prime },z^{\prime }, u^{\prime },\hat{y})|\leq L \left( |y-y^{\prime
}|+|z-z^{\prime }|+|u-u^{\prime }| \right),\smallskip \\ 
\multicolumn{1}{l}{\left( iii\right)} & \displaystyle|f(t,\phi ,y,z,u,\hat{y})-f(t,\phi ,y,z,u,\hat{y}^{\prime })|^{2} \leq K\int_{-\delta }^{0}\left(
\left\vert \hat{y}(\theta )-\hat{y}^{\prime }(\theta )\right\vert
^{2} \right) \alpha (d\theta ),\smallskip \\ 
\left( iv\right) & \displaystyle\left\vert f\left( t,\phi ,0,0,0,0\right)
\right\vert <M(1+\left\Vert \phi \right\Vert _{T}^{p}).%
\end{array}%
\end{equation*}

\item[($A_5$)] The function $f\left( \cdot ,\cdot ,y,z,u,
\hat{y} \right) $ is $\mathbb{F}$--progressively measurable, for any $%
\left( y,z,u,\hat{y}\right) \in \mathbb{R} \times \mathbb{R}%
^{d^{\prime }} \times \mathbb{R} \times L^{2}\left( [-\delta ,0];\mathbb{R} \right)$.

\item[($A_6$)] The function $\phi \mapsto h(t,\phi)$  is continuous and $%
|h(\phi )|\leq M(1+\left\Vert \phi \right\Vert _{T}^{p}),$ for all $\phi \in 
{\Lambda }.$
\end{itemize}
\end{ass}

The following remark generalizes a classical result, see, e.g., Theorem 14.1.1 in \cite{Delong1}, Theorem 2.1 in   \cite{DelongImkeller} or Theorem 2.1 in \cite{DelongImkeller2}.

\begin{remark}
\label{small}
In order to show both existence and uniqueness of a
solution to the backward part of the system \eqref{fbsde} and to obtain the
continuity of $Y^{t,\phi }$ with respect to $\phi $, 
we need  to impose
$K$ or $\delta$ to be small enough.
More precisely, we will assume that there exists a constant $\chi \in
(0,1) $, such that:%
\begin{equation}
K\,\frac{\chi e^{\big(\chi +\frac{6L^{2}}{\chi }\big)\delta }}{%
(1-\chi )L^{2}}\,\max \left\{ 1,T\right\} <\frac{1}{578}\,.
\label{condition_KT}
\end{equation}
\end{remark}

The main difference between our result and Theorem 3.4 in \cite{CDMZ}
relies in the presence of a jump component in the dynamics of the unknown process $Y^{t,\phi}$: this further term implies a stronger bound in the condition enforced in Eq. \eqref{condition_KT}.

Hence, if $K$ or $\delta $ are be small enough to satisfy the condition stated in Eq. \eqref{condition_KT}, then there exists a unique solution of \eqref{backward} and the following theorem holds

%

\begin{theorem}
\label{theorem1}
Let assumptions $A_4, A_5, A_6$ hold. If condition \eqref{condition_KT} is satisfied, then there exists a unique solution $(Y^{t,\phi}, Z^{t,\phi}, U^{t, \phi})$ of the BSDE  described in \eqref{backward} such that $(Y^{t,\phi}, Z^{t,\phi}, U^{t, \phi}) \in \mathbb{S}_t^2 (\mathbb{R}) \times  \mathbb{H}^2_t (\mathbb{R}^{ l}) \times  \mathbb{H}_{t,N}^2 (\mathbb{R}) $ for all $t \in [0,T]$ and the application $t \rightarrow (Y^{t,\phi}, Z^{t,\phi}, U^{t, \phi})$ is continuous from $[0,T]$ into $ \mathbb{S}_0^2 (\mathbb{R}) \times  \mathbb{H}^2_0 (\mathbb{R}^{l})  \times \mathbb{H}_{0,N}^2 (\mathbb{R})$.
\end{theorem}

The proof of Th. \ref{theorem1} is provided in Appendix \ref{AppA} and it is mainly based on the Banach fixed point theorem.


We emphasize that similar results hold also for multi-valued processes, namely $Y: \Omega \times [0,T] \rightarrow \mathbb{R}^m$, $Z: \Omega \times [0,T] \rightarrow \mathbb{R}^{m \times l}$ and $U: \Omega \times [0,T] \times (\mathbb{R}^n \setminus \{0\}) \rightarrow \mathbb{R}^m$.
Further difficulties may arise, due to presence of correlation between the different components of $Y \in \mathbb{S}_t^2 (\mathbb{R}^m)$ or the necessity of introducing the {\it $n$-fold iterated stochastic integral}, see \cite{Bell}, \cite{be} or \cite{Oliva} (Sec. 2.1),  for further details.


%
%
%
%
%
%
%

\section{The Feynman-Kac formula}
\label{sec4}

In what follows we prove that the solution of of Eq. \eqref{fbsde}, namely the path-dependent forward-backward system with delayed generator $f$ and driven by a Lévy process,
can be connected to solution of path-dependent PIDE represented by the non-linear Kolomogorov equation \eqref{Kolmogorov}.

\subsection{The Delfour-Mitter space}

According to \cite{Banos}, we need the solution of the forward SDE \eqref{fbsde} to be a Markov process as to derive the Feynman-Kac formula.  The Markov property of the solution is fully known for the SDE without jumps, i.e. when $\gamma = 0$, see \cite{sea1} (Th. III. 1.1). Moreover, the Markov property  also holds, by enlarging the state space, for the solution  in a setting analogous to that of Eq. \eqref{forward}, see \cite{Oliva} (Prop. 2.6) where driving noises with independent increments are considered.
Since $X^{t,\phi} : [0,T] \times D ([0,T]; \mathbb{R}^d) \rightarrow \mathbb{R}^d$ in not Markovian, we enlarge the state space by considering the
process $X$ as a {\it process of the path}, by introducing a suitable Hilbert space, as described in \cite{FlandoliZanco} and \cite{MOTZ}, where they present a product-space reformulation of \eqref{forward2} splitting the present state $X (t)$ from the past trajectory $X_{[0,t)}$ by a particular choice of the state space. 

Accordingly, we enlarge the state space of our interest, starting  from paths defined on the Skorohod space $D \left( [0,T]; \mathbb{R}^d \right)$ to then consider a {\it new } functional space, the so-called {\it Delfour-Mitter space} $M^2 := L^2 ([-T,0]; \mathbb{R}^d) \times \mathbb{R}^d$, by exploiting the continuous embedding of $D\left( [0,T]; \mathbb{R}^d \right)$ into $L^2 ([-T,0]; \mathbb{R}^d)$, as in-depth analyzed in, e.g., \cite{Banos} and \cite{Oliva}. Moreover, we rewrite the second one as a function on $[-t, 0)$ via a change of variable (in time) and
then we lengthen it towards the past up to $[-T, 0)$ to consider the whole path.
The latter approach allows us to identify the càdlàg process $X^{t,\phi}$ with the enlarged process ${X}^{t, \eta, x}$ defined over $M^2$, namely
$${X}^{t, \eta, x}= \left( \begin{array}{c} X^{t, \eta, x} (s), \{X^{t, \eta, x}_{s+r}\}_{r \in [-T,0]}
\end{array} \right) \in \mathbb{R}^d \times {L^2}([-T, 0];\mathbb{R}^d) \, . $$
Here, $X(s)$ denotes the present state that is well defined for any time $s \in [0,T]$, $X_{s+r}$ denoting the {\it segment} of the path described by $X$, which takes values in $L^2([-T, 0];\mathbb{R}^d)$, during the time interval $[s - T, s]$.

It is worth mentioning that $M^2$ has a Hilbert space structure, endowed with the following scalar product
$$\langle \phi  , \psi  \rangle_{M^2} = \langle \phi, \psi \rangle_{L^2} + \phi(0) \cdot \psi(0) \, , $$
with associated norm
$$ || \phi  ||^2_{M^2} = || \phi ||_{L^2} + |\phi(0)|^2 \, ,$$
where $\cdot$ and $|\cdot|$ stand for the scalar product in $\mathbb{R}^d$, resp. for the euclidean norm in $\mathbb{R}^d$, while  $\langle \cdot, \cdot \rangle_{L^2}$, resp. $||\cdot||_{L^2}$, indicates the scalar product, resp. the norm in $L^2 := L^2 (-T, 0; \mathbb{R}^d$).

By the continuous injections $D \subset M^2$,
we can formulate the forward equation with \eqref{forward} $M^2$ coefficients
\begin{equation}
\label{forward2}
\begin{cases}
 \displaystyle X^{t, \eta, x} (s) = x +  \int_t^s \tilde{b} (r, X^{t, \eta, x}_r, X^{t, \eta, x}(r) ) dr +  \int_t^s \tilde{\sigma} (r, X^{t, \eta, x}_r, X^{t, \eta, x}(r)) d W (r) + \smallskip \\
 \displaystyle \qquad + \int_t^s  \int_{\mathbb{R}_0} \tilde{\gamma} (r, X^{t, \eta, x}_r, X^{t, \eta, x}(r),z) \tilde{N} (dr, dz) \, ,  \displaystyle \hspace{1cm}  s \in [t,T] \smallskip\\
\left(  X^{t, \eta, x}_s , X^{t, \eta, x}(s) \right)= (\eta, x) \, ,  \hspace{4.7cm}   s \ \in [0,t] \, . \smallskip
\end{cases}
\end{equation}

For $t\in\lbrack0,T]$, $\phi\in\Lambda$ and $\varphi\in L^{2}%
([-T,0];\mathbb{R}^{d})$, let us set:

\begin{itemize}
\item compatible initial conditions $ x^{t,\phi} := \phi(t)$ and $\eta^{t,\phi}\in D([-T,0];\mathbb{R}^{d})$
defined by%
\begin{equation}
\label{eta}
\eta^{t,\phi}(\theta):=\left\{
\begin{array}
[c]{ll}%
\phi(t+\theta), & \theta\in\lbrack-t,0]\text{;}\\
\phi(0), & \theta\in\lbrack-T,t)\text{;}%
\end{array}
\right.
\end{equation}

\item $\varphi^{t}\in\Lambda$ defined by%
\begin{equation}
\label{varphi}
\varphi^{t}(\theta):=\left\{
\begin{array}
[c]{ll}%
\varphi(\theta-t), & \varphi\in D([-T,0];\mathbb{R}^{d}),\ \theta\in
\lbrack0,t]\text{;}\\
\varphi(0), & \varphi\in D([-T,0];\mathbb{R}^{d}),\  \theta\in(t,T]\text{;}\\
0, & \varphi\notin D([-T,0];\mathbb{R}^{d})\text{;}%
\end{array}
\right.
\end{equation}

\item $\tilde{b}:[0,T]\times M^{2}\rightarrow\mathbb{R}^{d}$, $\tilde{\sigma
}:[0,T]\times M^{2}\rightarrow\mathbb{R}^{d\times l}$, $\tilde
{\gamma}:[0,T]\times M^{2}\times\mathbb{R}\rightarrow\mathbb{R}^{d}$ defined
by%
\begin{equation}
\label{enlarged}
\tilde{b}(t,\varphi,x):=b(t,\varphi^{t})\text{, }   \quad \tilde{\sigma}(t,\varphi
,x):=\sigma(t,\varphi^{t})\text{, } \quad \tilde{\gamma}(t,\varphi,x,z):=\gamma
(t,\varphi^{t},z);
\end{equation}

\end{itemize}


where $b:[0,T]\times\Lambda\rightarrow\mathbb{R}^{d}$,\ $\sigma:[0,T]\times
\Lambda\rightarrow\mathbb{R}^{d \times l}$, $\gamma:[0,T]\times
\Lambda\times\mathbb{R}\rightarrow\mathbb{R}^{d}$ and $f$ are the given coefficients of Eq. \eqref{forward}.

We have to additionally impose that $b$, $\sigma$, $\gamma$, $f$ and $h$ are
locally Lipschitz-continuous with respect to $\phi\in\Lambda$ in the $L^{2}$-norm: as understood in \cite{generalized}, the function $u$ in Eq. \eqref{kac}, namely the solution of the Kolmogorov Equation \eqref{Kolmogorov}, is locally Lipschitz in $x$ with polynomial growth. Thus, in order to have the same regularity for the solution of the BSDE system with forward Eq. \eqref{forward2}, we require that the coefficients $b$, $\sigma$ and $\gamma$ to be locally Lipschitz.

\begin{ass}
\label{ass6.2}

There exists $K \geq 0$ and $m \geq 0$ such that:

\begin{itemize}
\item[($A_7$)] for all $t \in [0,T]$ and for all $\phi_1, \phi_2 \in \Lambda$, 
    $$ \begin{array}{l}
\displaystyle | b (t, \phi_1) - b (t, \phi_2) |^2  + | \sigma (t, \phi_1) - \sigma (t, \phi_2) |^2  + \int_{\mathbb{R}_0}  | \gamma (t, \phi_1,z) - \gamma (t, \phi_2, z) |^2 \nu (dz)
 \smallskip \\
\displaystyle \qquad   \leq K || \phi_1 - \phi_2||^2_{L^2} (1 + | |\phi_1||^2_{L^2} + ||\phi_2||^2_{L^2} ) \, ;
 \end{array}
$$
\item[($A_8$)]  for all $t \in [0,T]$, $y \in \mathbb{R}$, $z \in\mathbb{R}^l$, $u \in \mathbb{R}$, $\hat{y} \in L^2 ([-\delta,0] ; \mathbb{R})$ and for all $\phi_1, \phi_2 \in \Lambda$ 
$$
\begin{array}{l}
  |f (t, \phi_1, y, z, u, \hat{y})   - f (t, \phi_2, y, z, u, \hat{y} )| \leq \smallskip \\
 \qquad \qquad  \qquad  K (1 + ||\phi_1||_{L^2} + ||\phi_2||_{L^2} + |y|)^m \cdot (1 + |z| + |u|) || \phi_1 - \phi_2||_{L^2} \, .
\end{array}
$$
\end{itemize}

\end{ass}

Within this setting, lifting the state space turns out to be particularly convenient to
investigate differentiability properties of the solution and to relate the solution of Eq. \eqref{forward2}  (combined with a proper backward equation) to the solutions of the non-linear Kolmogorov equation defined by Eq. \eqref{Kolmogorov}, defined on $[0,T] \times  \Lambda$.

\begin{remark}
\label{weak}
It is also possible to work directly on the Skorohod space $D$. However, since $D$ is not a separable Banach space, one has to consider weaker topologies on $D$, following a semi-group approach like the one developed by Peszat and Zabczyk in \cite{Peszat}.
\end{remark}

\subsection{Main theorem}

The main result of this section provides a nonlinear version of the Feynman–Kac formula in the case where the process $X^{t,\phi}$ has jumps and the generator of the backward dynamic $f$ depends on the past values of $Y$.

\begin{theorem}[Feynman-Kac formula]
\label{th5.1}
Let hypotheses $(A_1)-(A_8)$ and Assumption \ref{ass6.2} hold with condition \eqref{condition_KT} being verified. Then
\begin{equation}
\label{kac}
Y^{t,\phi} (s)= u (s, X^{t,\phi}) \, , \hspace{.8cm} \textsl{for all s }\in [0,T] ,
\end{equation}
for any $(t,\phi) \in [0,T] \times \Lambda $ , where $Y^{t, \phi}$ is the solution of the following BSDE
\begin{equation}
\label{backward_delayinY}
\begin{cases}
\displaystyle Y^{t, \phi} (s) = h   (X^{t, \phi})   + \int_s^T f \big( r,  X^{t, \phi}, Y^{t, \phi} (r), Z^{t, \phi} (r), \tilde{U}^{t, \phi} (r), Y_r^{t,\phi} \big) dr  \smallskip \\ 
\displaystyle  \qquad 
 -  \int_s^T  Z^{t, \phi} (r) dW(r)  -  \int_s^T  \int_{\mathbb{R}_0} U^{t, \phi} (r, z) \tilde{N} (dr, dz) \,  , \hspace{.3cm}  s \in [t,T] \smallskip \\ 
\displaystyle
Y^{t,\phi} (s) = Y^{s,\phi} (s) \, , \hspace{.3cm} Z^{t,\phi} (s) = U^{t,\phi} (s,z) = 0 \, , \hspace{2.2cm}  s \in [0,t] \, ,
\end{cases}
\end{equation}
and $u(t, \phi): [0,T] \times \Lambda \rightarrow \mathbb{R}$ is a deterministic function defined by
\begin{equation}
  \label{FK}
 u (t, \phi) = Y^{t,\phi} (t) \quad (t, \phi) \in [0,T] \times \Lambda \, .  
\end{equation}

Moreover, the solution of the forward backward equation \eqref{fbsde} is the quadruple $(X, Y, Z, U)$ taking values in $\Lambda \times \mathbb{R} \times \mathbb{R}^{ l} \times \mathbb{R}$
\end{theorem}

To prove the representation formula \eqref{FK}, we adapt the proof of Theorem 4.10 of \cite{CDMZ} by adding the contribution of $U$ and $\tilde{U}$, respectively modelling the process and the integral term connected to the jump component.

\begin{proof}

We follow the Picard iteration scheme, hence considering the iterative process of the BSDE with delayed generator driven by Lévy process described by
\begin{align*}
& Y^{n+1, t, \phi} (s) =   h (X^{t, \phi} )  \\
& \qquad + \mathlarger{\int}_s^T f \Big( r, X^{t, \phi}, Y^{n+1, t, \phi} (r), Z^{n+1, t, \phi} (r), \int_{\mathbb{R}} U^{n+1, t, \phi} (r,z) \lambda (z) \nu (dz), Y_r^{n,t,\phi} \Big) dr \\ 
& \qquad  -  \mathlarger{\int}_s^T  Z^{n+1, t, \phi} (r) dW(r) -  \mathlarger{\int}_t^T  \mathlarger{\int}_{\mathbb{R}} U^{n+1,t,\phi} (r,x) \tilde{N} (dr, dz) \, ,
\end{align*} 
with $Y^{0,t,\phi} \equiv  0 $, $Z^{0,t,\phi} \equiv  0 $ and $U^{0,t,\phi} \equiv  0 $.

Let us suppose that there exists a $\mathbb{F}$-progressively measurable functional $u_n: [0,T] \times \Lambda \rightarrow \mathbb{R}$ such that $u_n$ is locally Lipschitz and $Y^{n, t, \phi} (s) = u_n (s, X^{t, \phi})$ for every $t,s \in [0,T]$ and $\phi \in \Lambda$.

According to Theorem 4.10 in \cite{CDMZ}, we consider the delayed term
$$ Y_r^{n, t, \phi} (s)   = \big( Y^{n, t, \phi} (r + \theta) \big)_{\theta \in [-\delta, 0]}.$$

Since $Y^{n, t, \phi} (r + \theta) = u_n (r + \theta, X^{t, \phi} )$ if $r+\theta \geq 0$ and $Y^{n, t, \phi} (r + \theta) = Y^{n, t, \phi} (0) = u_n (0, X^{t, \phi})$ if $r + \theta < 0$, by defining
$$\tilde{u}_n (t, \phi) := \left( u_n (\mathds{1}_{[0,T]}(t + \theta) , \phi) \right)_{\theta \in [-\delta, 0]} \, , $$
the delayed term reads
$$Y_r^{n,t,\phi} = \tilde{u}_n (r, X^{t, \phi} )$$
and our equation becomes
\[%
\begin{cases}
\displaystyle X^{t,\phi}(s)=\phi(t)+\int_{t}^{s}b(r,X^{t,\phi})dr+\int_{t}%
^{s}\sigma(r,X^{t,\phi})dW(r)+\\
\qquad\displaystyle+\int_{t}^{s}\int_{\mathbb{R}\setminus\{0\}}\gamma
(r, X^{t,\phi},z)\tilde{N}(dr,dz)\,,\hspace{3.5cm}s\in\lbrack t,T]\\
X^{t,\phi}(s)=\phi(s)\,,\hspace{7.5cm}s\ \in\lbrack0,t]\smallskip\\
\displaystyle Y^{t,\phi}(s)=h(X^{t,\phi})+\int_{s}^{T}f\left(  r,X^{t,\phi
},Y^{t,\phi}(r),Z^{t,\phi}(r),\tilde{U}^{t,\phi}(r),\tilde{u}_{n}(r,X^{t,\phi})\right)  dr \smallskip \\
\qquad\displaystyle-\int_{s}^{T}Z^{t,\phi}(r)dW(r)-\int_{s}^{T}\int%
_{\mathbb{R}\setminus\{0\}}U^{t,\phi}(r,z)\tilde{N}(dr,dz)\,,\hspace
{0.3cm}s\in\lbrack t,T] \smallskip \\
Y^{t,\phi}(s)=Y^{s,\phi}(s)\,,\hspace{0.3cm}Z^{t,\phi}(s)=U^{t,\phi
}(s,z)=0\,,\hspace{2.5cm}s\in\lbrack0,t]\, .
\end{cases}
\]

We fix $n$ and we define $\psi:[0,T]\times M^{2}\times\mathbb{R}\times\mathbb{R}^{l%
}\times\mathbb{R}
\rightarrow\mathbb{R}$ as%
\begin{equation}
\label{psi}
\psi (t,\varphi,x,y,z,u):=f(t,\varphi^{t},x,y,z,u,\tilde{u}%
_{n}(r,\varphi^{t})) \, ,
\end{equation}

where $\tilde{h}:M^{2}\rightarrow\mathbb{R}$ is defined as follows%
$$
\tilde{h}(\varphi,x):=h(\varphi^{T}) \, .
$$
Since $u_{n}$ is locally
Lipschitz-continuous, one can show that $\psi$ is also locally Lipschitz in $\varphi$.

After fixing $\eta=\eta^{t,\phi}$, $x=x^{t,\phi}$ as in Eq. \eqref{eta}, we enlarge the state space for  in an equivalent way as%
\[%
\begin{cases}
\displaystyle \mathrm{d}X^{\tau,\eta,x}(t) & =\tilde{b}(t,X_{t}^{\tau,\eta,x},X^{\tau,\eta
,x}(t))\mathrm{d}t+\tilde{\sigma}(t,X_{t}^{\tau,\eta,x},X^{\tau,\eta
,x}(t))\mathrm{d}W(t)\\
\displaystyle
& \quad+\int_{\mathbb{R}_{0}}\tilde{\gamma}(t,X_{t}^{\tau,\eta,x},X^{\tau
,\eta,x}(t),z)\tilde{N}(\mathrm{d}t,\mathrm{d}z)\\ \displaystyle
\smallskip(X_{\tau}^{\tau,\eta,x},X^{\tau,\eta,x}(\tau)) & =(\eta,x)\\ \displaystyle
\mathrm{d}Y^{\tau,\eta,x}(t) & =\psi \left(  t,X_{t}^{\tau,\eta,x},X^{\tau
,\eta,x}(t),Y^{\tau,\eta,x}(t),Z^{\tau,\eta,x}(t),\tilde{U}^{\tau,\eta
,x}(t)\right)  \mathrm{d}t\\ \displaystyle
& \quad+Z^{\tau,\eta,x}(t)\mathrm{d}W(t)+\int_{\mathbb{R}_{0}}U^{\tau,\eta
,x}(t,z)\tilde{N}(\mathrm{d}t,\mathrm{d}z)\\
Y^{\tau,\eta,x}(T) & =\tilde{h}(X_{T}^{\tau,\eta,x},X^{\tau,\eta,x}(T)) \, , 
\end{cases}
\]
where the solutions of the two formulations $\left(  X^{\tau,\eta,x},Y^{\tau,\eta,x},Z^{\tau,\eta,x}\right)$ and $\left(
X^{t,\phi},Y^{t,\phi},Z^{t,\phi}\right)  $ coincide. 

We are assuming that enlarged coefficients $\tilde{b}$, $\tilde{\sigma}$,
$\tilde{\gamma}$ defined accordingly to Eq. \eqref{enlarged} satisfy conditions (A1)-(A2), see \cite{Oliva}, pp. 8-9,
 while $\psi$ and $\tilde{h}$ satisfies, respectively, (B1) and (B2), see \cite{Oliva}, p.24.
Then, exploiting Theorem 4.5 in \cite{Oliva}, we know that there exists a locally Lipschitz function $\tilde{u}:[0,T]\times M^{2}%
\rightarrow\mathbb{R}$ and a representation formula such that%
\[
Y^{\tau,\eta,x}(t)=\tilde{u}\left(  t,X_{t}^{\tau,\eta,x},X^{\tau,\eta,x}(t)\right)  .
\]
Let us define the following locally Lipschitz, non anticipative, functional $u_{n+1}:[0,T]\times\Lambda\rightarrow\mathbb{R}$ by%
\[
u_{n+1}(t,\phi):=\tilde{u}(t,\phi_{t},\phi(t)) \, ,
\]
where the time shifting $\phi_t $ is defined accordingly to Eq. \eqref{eta}, then:
\[
Y^{n+1,t,\phi}(s)=u_{n+1}(s,X^{t,\phi}),\ \forall s\in\lbrack0,T]\,.
\]

Notice that $\left( Y^{n,t, \phi} , Z^{n,t, \phi}, U^{n,t, \phi} \right)$ is the Picard iteration needed to construct the solution $\left( Y^{t, \phi} , Z^{t, \phi}, U^{t, \phi} \right)$ such that 
\begin{equation*}
\left( Y^{n+1,\cdot ,\phi },Z^{n+1,\cdot ,\phi }, U^{n+1,\cdot ,\phi }\right) =\Gamma (Y^{n,\cdot
,\phi },Z^{n,\cdot ,\phi }, U^{n,\cdot ,\phi })\,,
\end{equation*}%
where $\Gamma $ is the contraction defined in the proof of Theorem \ref{theorem1}. By applying Theorem \ref{theorem1}, we then have:%
\begin{equation*}
\lim_{n\rightarrow \infty }\mathbb{E}\big(\sup_{s\in \lbrack 0,T]}\big|%
Y^{n,t,\phi }(s)-Y^{t,\phi }(s)\big|^{2}\big)=0.
\end{equation*}%
Of course, $u_{n}(t,\phi )$ converges to $u(t,\phi ):=\mathbb{E} \, [ Y^{t,\phi
}(t)],$ for every $t\in \lbrack 0,T]$ and $\phi \in {\Lambda }$, hence implying  that the nonlinear Feynman--Kac formula $Y^{t,\phi }\left( s\right)
=u(s,X^{t,\phi })$ holds.\hfill $\smallskip $
\end{proof}

\section{Mild solution of the Kolmogorov Equation}
\label{sec5}

In this section, we prove the existence of mild solution of the Path-dependent Partial Integro-Differential Equation (PPIDE) Kolmorogov equation \eqref{Kolmogorov} showing a dependence both on a delayed term and on integral term modelling jumps. For the sake of completeness, let us 
start recalling the following notion of {\it mild solution}.

\begin{defn}
\label{mild_def}
A function $u:[0,T] \times M^2 \rightarrow \mathbb{R}$ is a mild solution to Eq. \eqref{Kolmogorov} if there exists $C>0$ and $m\geq 0$ such that for any $t \in [0,T]$ and any $\phi_1, \phi_2 \in M^2$, $u$ satisfies
\begin{equation}
\begin{array}{l}
| u(t, \phi_1) - u (t, \phi_2) | \leq C || \phi_1 - \phi_2||_{M^2} (1 + || \phi_1||_{M^2} + ||\phi_2||_{M^2} )^m \, ; \smallskip 
\\ 
\displaystyle  |u (t,0,0| \leq C \, ;
\end{array}
\label{mild1}
\end{equation}%
and the following equality holds true
\begin{equation}
\label{semigroup}
u(t,\phi) = P_{t,T} h(\phi)  + \int_t^T P_{t,s} \left[ 	f \left( \cdot, u(s, \cdot), \partial_x u (s, \cdot) \sigma(s, \cdot), (u(\cdot, \phi))_s), \mathcal{J} u (s, \cdot) \right)  (\phi) ds \right] \, ,
\end{equation}
for all $t \in [0,T]$ and $\phi \in M^2$, $(u(\cdot, \phi))_s)$ being the delayed term defined in Eq. \eqref{delayed}.
$\left( P_{t,s} \right)_{0 \leq t \leq s \leq T}$ corresponds the {\it Markov transition semigroup} corresponding
to the operator $\mathcal{L}$ introduced in Eq. \eqref{operatorL}, that in the lifted setting corresponds to the generator of the Markov semigroup associated to  forward equation appearing in Eq. \eqref{forward2}, for more details, see, e.g.,  \cite{generalized}).
\end{defn}

From \cite{Oliva} (Prop. 2.6), we know that the strong solution $X^{t,\phi} \in M^2$ of Eq. \eqref{forward2} is a {\it Markov process} in the sense that 
$$\mathbb{P} ( (X^{t,\phi} \in B | \mathcal{F}_s) = \mathbb{P} (( (X^{t,\phi} \in B | (X^{t,\phi} (s) = \phi (s) ) , \quad \mathbb{P}-a.s. \, ,$$
for all $0 \leq s \leq  T$ and for all Borel sets $B \in \mathcal{B} (M^2)$, the proof of this fundamental result being stated in several works, as, e.g. \cite{Banos} (Th. 3.9), \cite{Reib} (Prop. 3.3) or \cite{Peszat} (Sec. 9.6).

The next theorem represents the core result of this section.

\begin{theorem}[Existence]
\label{Kolmind}
Let Assumptions $(A_4) -  (A_8)$ hold true. Then, the path-dependent
partial integro-differential Eq. \eqref{Kolmogorov} admits a mild solution
$v$, in the sense of Definition \ref{mild_def}.
In particular, 
 we identify $v$ with $u$, i.e. the restriction over the set of càdlàg paths of the mild solution $v$ of Eq. \eqref{Kolmogorov}, namely 
\begin{equation}
    \label{restriction}
 v (t,\varphi,x):= u (t,\varphi^{t}) \, ,\end{equation}
where $\varphi^t$ is defined according to  Eq. \eqref{varphi} and $u$ coincides
with Eq. \eqref{FK}.
\end{theorem}

\begin{proof}
Let us consider the backward component of the  FBSDE described in Eq. \eqref{fbsde} for $s \in [t,T]$

\begin{small}
\begin{equation}
\label{proof1}
\begin{array}{l}
    \displaystyle Y^{t,\phi}(s)=h(X^{t,\phi})+\int_{s}^{T}f\left(  r,X^{t,\phi
},Y^{t,\phi}(r),Z^{t,\phi}(r),\int_{\mathbb{R}_{0}}U^{t,\phi}(r,z)\lambda
(z)\nu(dz),Y_r^{t, \phi})\right)  dr\\
    \displaystyle \qquad\displaystyle-\int_{s}^{T}Z^{t,\phi}(r)dW(r)-\int_{s}^{T}\int%
_{\mathbb{R}\setminus\{0\}}U^{t,\phi}(r,z)\tilde{N}(dr,dz)\,,\hspace
{2cm}s\in\lbrack t,T] \, .
\end{array}
\end{equation}
\end{small}

 By Theorem \ref{theorem1}, there exists a representation formula for the solution of the BSDE \eqref{proof1}. Moreover, by means of Eq. \eqref{FK}, we can write the delayed term $Y_r$ as a function of the path of  solution of the forward dynamic $X^{t, \phi}$ and, thus, we obtain
$$\tilde{u} (t, \phi) := \left( u (\mathds{1}_{[0,T]}(t + \theta) , \phi) \right)_{\theta \in [-\delta, 0]} \, ,$$

that we can plug into \eqref{proof1}, leading to

\[%
\begin{cases}
\displaystyle Y^{t,\phi}(s)=h(X^{t,\phi})+\int_{s}^{T}f\left(  r,X^{t,\phi
}, Y^{t,\phi}(r),Z^{t,\phi}(r),\tilde{U}^{t,\phi}(r), \tilde{u} (r,X^{t,\phi}) \right)  dr\\
\qquad\displaystyle-\int_{s}^{T}Z^{t,\phi}(r)dW(r)-\int_{s}^{T}\int%
_{\mathbb{R}\setminus\{0\}}U^{t,\phi}(r,z)\tilde{N}(dr,dz)\,,\hspace
{2cm}s\in\lbrack t,T]\\
Y^{t,\phi}(s)=Y^{s,\phi}(s)\,,\hspace{1.5cm}Z^{t,\phi}(s)=U^{t,\phi
}(s,z)=0\,,\hspace{3cm}s\in\lbrack0,t],
\end{cases} \, .
\]

At this point, we enlarge the state space going through $M^2$ coefficients analogously to the proof of Theorem \ref{theorem1}, obtaining
\begin{equation}
\begin{array}{l}
\displaystyle
Y^{t, \eta, x} (s) = h (X_T^{s,\eta,x}, X ^{s,\eta,x} (T)) \\
\displaystyle \qquad +  \int_s^T \psi \big( r, X_r^{t, \eta, x},  X^{t, \eta, x} (r) , Y^{t, \eta, x} (r), Z^{t, \eta, x} (r), \tilde{ U}^{t, \eta, x} (r),  \big) dr  \\ 
 \displaystyle \qquad  -  \int_s^T  Z^{t, \eta, x} (r) dW(r) -  \int_t^T  \int_{\mathbb{R}_0} U^{t, \eta, x} (r, z) \tilde{N} (dr, dz) \,  , 
 \end{array}
\end{equation}
recalling the map $\psi$ defined in Eq. \eqref{psi}, we may obtain a BSDE in the same setting of \cite{Oliva}. Taking the expectation and exploiting the representation formulas described by Theorem 4.5 in \cite{Oliva}, we know that there exists a process
$$
Y^{t, \eta, x} (s) \, =  \quad v (s, X_s^{t, \eta, x}, X^{t, \eta, x}(s) ) 
$$
with $v: [0,T] \times L^2 ([0,T], \mathbb{R}^d) \times \mathbb{R}^d \rightarrow \mathbb{R}$, and we conclude identifying  $u$ with the restriction of $v$ on the set of càdlàg paths $\Lambda$.



\end{proof}

\begin{remark}[Uniqueness]
Let us underline that the aforementioned lack of uniqueness can be overcome by considering the problem of uniqueness of the mild solution in the case of $f$ being independent of the term $%
\left( u(\cdot ,\phi )\right) _{t}\,$, with passages similar to the proof of \ref{th5.1}.
Indeed, let us take two mild solutions $u^{1}$ and $u^{2}$ of the
path-dependent PDE \eqref{Kolmogorov}. We define%
\begin{equation*}
f^{i}(t,\phi ,y,z,w):=f(t,\phi ,y,z,w,\left( u^{i}\left( \cdot ,\phi \right)
\right) _{t})\,,\quad i=\overline{1,2}\,.
\end{equation*}%
Using these drivers we can consider the following BSDEs:%
\begin{equation}
\begin{array}{l}
\displaystyle Y^{t,\phi }\left( s\right) =h(X^{t,\phi })+\int_{s}^{T}f^{i}(r,X^{t,\phi
},Y^{t,\phi }\left( r\right) ,Z^{t,\phi }\left( r\right), \tilde{U}^{t,\phi }\left( r\right)
)dr-\int_{s}^{T}Z^{t,\phi }\left( r\right) dW\left( r\right) \smallskip \\ \qquad \displaystyle - \int_s^T \int_{\mathbb{R}_0} U^{t, \phi } (r,z) \tilde{N} (dr, dz) \smallskip \,,\quad i=%
\overline{1,2}\,, 
\end{array}
\label{BSDE_delayed 3}
\end{equation}%
for which there exist unique solutions $\left( Y^{i,t,\phi }\, , \, Z^{i,t,\phi
} \, , U^{i,t,\phi
}  \,  \right) \in
\mathbb{S}_t^2 (\mathbb{R}) \times \mathbb{H}^2_t (\mathbb{R}^{l})\times \mathbb{H}_{t,\nu}^2  (\mathbb{R}) $
for
$ i=\overline{1,2}\,.$

By Theorem \ref{th5.1} we see that%
\begin{equation*}
Y^{i,t,\phi }\left( s\right) =v^{i}(s,X^{t,\phi }),\quad \text{for all }s\in %
\left[ 0,T\right] ,\quad \text{a.s.,}
\end{equation*}%
for any $\left( t,\phi \right) \in \left[ 0,T\right] \times \Lambda %
,$ where $v^{i}:\left[ 0,T\right] \times \Lambda \rightarrow 
\mathbb{R}$ is (the restriction) of the solution of \eqref{FK}, i.e.%
\begin{equation*}
v^{i}(t,\phi ):=Y^{i,t,\phi }\left( t\right) ,\;(t,\phi )\in \lbrack
0,T]\times \Lambda \,.
\end{equation*}%
Hence, by Theorem \ref{Kolmind}, we obtain that the functions $v^{i}$
are solutions of the PDE of type \eqref{Kolmogorov}, but without the delayed terms 
$\left( v^{i}\left( \cdot ,\phi \right) \right) _{t}:$%
\begin{equation}
\left\{ 
\begin{array}{r}
-\partial _{t}v^{i}(t,\phi )-\mathcal{L}v^{i}(t,\phi )-f^{i}(t,\phi
,v^{i}(t,\phi ),\partial _{x}v^{i}\left( t,\phi \right) \sigma (t,\phi
), \mathcal{J} v^i (t, \phi))=0,\smallskip \\ 
\multicolumn{1}{l}{v^{i}(T,\phi )=h(\phi ),\quad i=\overline{1,2}\,.}%
\end{array}%
\right.  \label{PDKE 3}
\end{equation}%
Since $u^{i}$ is also solution to equation \eqref{Kolmogorov}, by using Theorem
4.5 from \cite{Oliva} we know that%
\begin{equation*}
u^{i}(t,\phi )=v^{i}(t,\phi ),\quad \left( t,\phi \right) \in \left[ 0,T%
\right] \times \Lambda \,,\quad i=\overline{1,2}\,.
\end{equation*}%
By asking that $f$ satisfies assumption $(A_8)$,
we know that
\begin{equation*}
Y^{i,t,\phi }\left( s\right) =v^{i}(s,X^{t,\phi })=u^{i}(s,X^{t,\phi }),\ i=%
\overline{1,2}\,,
\end{equation*}
BSDEs (\ref{BSDE_delayed 3}) become a single equation,%
\begin{equation}
\begin{array}{l}
\displaystyle Y^{i,t,\phi }\left( s\right) =h(X^{t,\phi })+\int_{s}^{T}f(r,X^{t,\phi
},Y^{i,t,\phi }\left( r\right) ,Z^{i,t,\phi }\left( r\right)
, \tilde{U}^{i,t,\phi }\left( r\right), Y_{r}^{i,t,\phi })dr \smallskip \\ \qquad \displaystyle -\int_{s}^{T}Z^{i,t,\phi }\left( r\right) dW\left(r\right)  - \int_s^T \int_{\mathbb{R}_0} U^{t, \phi } (r,z) \tilde{N} (dr, dz)
 \,,  
\end{array}\label{BSDE_delayed 4}
\end{equation}%
with $i=\overline{1,2}\,,$ for which we have uniqueness from Theorem \ref{theorem1}.

Hence $Y^{1,t,\phi }=Y^{2,t,\phi }$ and therefore%
\begin{equation*}
u^{1}(t,\phi )=Y^{1,t,\phi }\left( t\right) =Y^{2,t,\phi }\left( t\right)
=u^{2}(t,\phi ).
\end{equation*}
\end{remark}

\section{Financial Application}

\label{SEC:FA}

In this section we provide a financial application moving from the model studied in, e.g.  \cite{CM}, or
\cite{EKPQ}. We consider a generalization of the so-called Large Investor Problem, where a so called \textit{%
large investor} wishes to invest on a given market, buying or selling a stock.  The investor has the peculiarity that his actions on the market can affect the stock price. We further generalize previous results in \cite{ CDMZ} \cite{EKPQ} or  by asking, in addition to a small time delay between the action of the large investor and the reaction of the market, the dynamic of the risky asset to be driven by a Poisson random measure.


More specifically, let us denote the investor's strategy by $\pi$ and the investment portfolio by $X^\pi$. Then, according to the theory previously developed, here $X^\pi$ is modeled as a càdlàg process and its past $X^\pi_r$ may affects directly the drift $\mu$ of the stock rate. Consequently we consider the following dynamics
\begin{equation}
\left\{ 
\label{lip1}
\begin{array}{l}
\displaystyle\frac{dS_0(t)}{S_0(t)}=r(t,X^\pi (t), X^\pi_t)dt\;, \smallskip \\ 
\displaystyle S_0(0)=1\;,\smallskip \\ 
\displaystyle\frac{dS(t)}{S(t)}=\mu \left( t,X^\pi (t), X^\pi_t \right) dt+\sigma (t)dW(t) + \int_{\mathbb{R}_0} \gamma (t,z) \tilde{N} (dt,dz) , \smallskip \\ 
\displaystyle  S(0)=s_{0}>0 \;,
\end{array}%
\right. 
\end{equation}%
where $r$, $\mu$, $\sigma$ and $\gamma$ are $\mathbb{F}^{W, \tilde{N}}$-predictable processes, being $\mathbb{F}^{W, \tilde{N}}$ the natural filtration associated to the Brownian motion $W$, that are required to be adapted to the Poisson random measure $\tilde{N}$, with compensator is defined according to Eq. \eqref{poisson}.
The initial datum $s_0$ belongs to the class of càdlàg stochastic process with values in $\mathbb{R}$.

The total amount of the portfolio of the large investor is described by
$$
\displaystyle dX^\pi (t)= \pi(t) \frac{dS(t)}{S(t)} + (X^\pi (t)- \pi(t))\frac{dS_0(t)}{S_0(t)}dt \, ,
$$
where $X^\pi(t)$ and $\pi(t)$ the present state of the processes with values in $\mathbb{R}$.

The goal is to find an admissible replicating strategy $\pi \in \mathcal{A}$ for a claim $\hat{F} (S(T))$.
We have that the portfolio $X$ evolves according to%
\begin{equation*}
\begin{array}{l}
\displaystyle dX^\pi(t)=\frac{\pi (t)}{S\left( t\right) }dS\left( t\right) +%
\frac{X^\pi\left( t\right) -\pi (t)}{S_0\left( t\right) }dS_0\left( t\right) \smallskip
\\ 
\displaystyle \qquad =\pi (t)\cdot \left[ \mu \left( t,X^\pi, X^\pi_r \right) dt+\sigma (t )dW(t) + \int_{\mathbb{R}_0} \gamma (t,z) \tilde{N}(dt, dz) \right] \smallskip
\\ 
\displaystyle \qquad  +\left[ X^\pi\left( t\right)
-\pi (t)\right] \cdot r(t,X^\pi (t), X_r^\pi )dt\,,%
\end{array}%
\end{equation*}%
where,  at any time $t\in \lbrack 0,T]$ by $%
\pi (t),$ represents the amount invested in the risky asset $S$, while by $X(t)-\pi (t),$
is the amount invested in the riskless bond $S_0$, and
$X^\pi(T)=h\left( S\right)$ represents the final condition.

Hence, for $t\in \left[ 0,T\right]$, we have
\begin{small}
\begin{equation}
\label{lip2}
\begin{array}{l}
\displaystyle X(t)=h\left( S\right) +\int_{t}^{T}F\left( s,X^\pi (s), X^\pi_s, \pi(s) \right) ds-\int_{t}^{T} Z(s)dW(s)  - \int_t^T \int_{\mathbb{R}_0} U (s,z)  \tilde{N} (ds, dz) ,  
\end{array}
\end{equation}
\end{small}
where, for the sake of readability, we denoted:%
\begin{small}
\begin{equation}
\label{lip_coeff}
    \begin{array}{l}
\displaystyle {F}\left( s,X^\pi (s) , X^\pi_s, \pi (s) \right) :=- r(s,X^\pi (s), X^\pi_s)
\left( X^\pi(s)-\pi (s)\right)  -\pi
(s) \mu \left( s,X^\pi (s), X^\pi_s \right)  \, ; \smallskip\\
Z(s) := \pi (s)\sigma
(s) \, ; \smallskip \\
U(s,z) := \pi(s) \gamma (s,z) \, . \smallskip
   \end{array}
\end{equation}%
\end{small}

We then impose that 
functions $r,\mu$, $\sigma$ and $\gamma$ are such that the function $F: [0,T]         \times \mathbb{R}         \times L^2 \left([-\delta, 0];     \mathbb{R}\right) \rightarrow \mathbb{R}$ satisfies assumptions $\mathrm{(A}_{4}\mathrm{)}$--$\mathrm{(A}_{6}\mathrm{).%
}$

Furthermore, we introduce a simplification to decouple the BSDE  \eqref{lip2} from the stock forward dynamic \eqref{lip1}, as to fit the setting of Theorem \ref{theorem1}. Indeed, instead of having $h(S)$  depending on the whole path of $S$, we 
consider the terminal function $\bar{h}$ that explicitly depends only on $W$, $\tilde{N}$, $Z$ and $U$. Thus, the functional $\bar{h}$ encodes the information at terminal time $T$ of the forward path of $S$ and it reads
\begin{equation*}
\bar{h}\left( W,X,Z, U\right) :=\tilde{h}(W,X,\sigma ^{-1}\left( \cdot
 \right), U)\,,
\end{equation*}%
while satisfying a Lipschitz condition:%
\begin{equation}
\begin{array}{l}
\left\vert \bar{h}\left( x,y,z, u\right) -\bar{h}\left( x,y^{\prime },z^{\prime
}, u^\prime \right) \right\vert ^{2}\leq \smallskip \\
\displaystyle \quad L\left[ \int_{0}^{T}\left\vert y\left(
s\right) -y^{\prime }\left( s\right) \right\vert
^{2}ds+\int_{0}^{T}\left\vert z\left( s\right) -z^{\prime }\left( s\right)
\right\vert ^{2}ds +\int_{0}^{T}\left\vert u\left( s\right) -u^{\prime }\left( s\right)
\right\vert ^{2}ds\right] .
\end{array}
\label{hlip}
\end{equation}

We underline that assuming suitable conditions
on $\mu $, $\sigma $ and $\gamma$, e.g. asking $\mu$ bounded and $\sigma $, $\gamma$ constant, then 
$\bar{h}$ satisfies $\mathrm{(A}_{6}\mathrm{)}$ and the
condition in Eq. \eqref{hlip}.

Therefore we can rewrite (\ref{lip2}) as%
\begin{small}
\begin{equation}
\begin{array}{l}
\displaystyle X(t)=\bar{h}\left( W,X,Z, U\right) +\int_{t}^{T}{F}\left( s, X(s) ,Z(s) ,X_{s}\right) ds-\int_{t}^{T}Z\left( s\right)
dW(s) \smallskip \\
\displaystyle \qquad - \int_t^T \int_{\mathbb{R}_0} U (s,z)  \tilde{N} (ds, dz) 
\quad t\in \left[ 0,T\right] \smallskip \, , \label{EQN:LargInvRe}
\end{array}
\end{equation}%
\end{small}
and we deduce from Theorem \ref{theorem1} that, under proper assumptions on
the coefficients, there exists a unique solution $\left( X,Z, U\right) $ to
equation \eqref{EQN:LargInvRe}.
Consequently, (\ref{lip2}) admits a unique solution $\left(
X,\pi\right) $, where 
\begin{equation*}
\pi\left( s\right) :=Z\left( s\right) \sigma^{-1}(s,X(s),X_{s})\, .
\end{equation*}

In order to obtain the connection with the associated PDE we first consider the
decoupled forward--backward stochastic system:%
\begin{small}
\begin{equation}
\left\{ 
\begin{array}{l}
\displaystyle\bar{W}^{t,\phi }\left( s\right) =\phi \left( t\right)
+\int_{t}^{s}dW\left( r\right) + \int_{\mathbb{R}} \int_t^s d\tilde{N}(r,z) ,\quad s\in \left[ t,T\right] ,\smallskip \\ 
\displaystyle\bar{W}^{t,\phi }\left( s\right) =\phi \left( s\right) ,\quad
s\in \lbrack 0,t),\smallskip \\ 
\displaystyle X^{t,\phi }\left( s\right) =\bar{h}(\bar{W}^{t,\phi
},X^{t,\phi },Z^{t,\phi }, U^{t, \phi})+\int_{s}^{T}\bar{F}(r,X^{t,\phi }\left( r\right)
,Z^{t,\phi }\left( r\right) ,X_{r}^{t,\phi })dr\smallskip \\ 
\displaystyle\quad \quad \quad \quad \quad \quad \quad \quad \quad -\int_{s}^{T}Z^{t,\phi }(r)dW\left( r\right)  - \int_s^T \int_{\mathbb{R}_0} U^{t, \phi} (r,z) \tilde{N} (dr, dz) 
, \quad s\in \left[ t,T\right] ,\smallskip \\ 
\displaystyle X^{t,\phi }\left( s\right) =X^{s,\phi }\left( s\right) ,\quad
\pi ^{t,\phi }\left( s\right) =0,\quad s\in \lbrack 0,t).%
\end{array}%
\right.
\label{lipbsde}
\end{equation}%
\end{small}
where the BSDE coefficients are defined according to \eqref{lip_coeff}.

If Assumptions \ref{ass2} hold for the coefficients, we can exploit \ref{theorem1} to derive that there exists a unique solution $\left( X^{t,\phi },\pi
^{t,\phi }\right) _{\left( t,\phi \right) \in \left[ 0,T\right] \times 
{\Lambda }}$ of the  system
\eqref{lipbsde}.

By Theorem \ref{th5.1}, for the solution 
$(X,Z,U)$ of the backward component in the above system, we can write that
for every $\left( t,\phi \right) \in \left[ 0,T\right] \times {%
\Lambda },$
\begin{equation*}
X^{t,\phi }\left( s\right) =u(s,\bar{W}^{t,\phi }),\quad \text{for all }s\in %
\left[ t,T\right] ,
\end{equation*}%
where, exploiting Theorem \ref{Kolmind}, $u\left( t,\phi \right) =X^{t,\phi }\left( t\right)$ is a specific restriction of the mild
solution, according to Eq. \eqref{restriction}, of the following path-dependent PDE:%
\begin{equation*}
\left\{ 
\begin{array}{l}
\displaystyle\partial _{t}u(t,\phi )+\frac{1}{2}\partial _{xx}^{2}u(t,\phi )+%
\bar{F}(t,u(t,\phi ),\partial _{x}u\left( t,\phi \right) ,\left( u\left(
\cdot ,\phi \right) \right) _{t})=0,\quad \phi \in \Lambda ,\;t\in
\lbrack 0,T),\smallskip \\ 
\displaystyle u(T,\phi )=\bar{h}(\phi ,\left( u\left( \cdot ,\phi \right)
\right) ),\quad \phi \in \Lambda.%
\end{array}%
\right.
\end{equation*}%


\section{Conclusions and Future Development}

The core result of this paper relies in deriving a stochastic representation for the solutions of a non-linear PDE and associating the solution of the PDE to a FBSDE with jumps and time-delayed generator. The presence of jumps both in the forward and in the backward dynamic and, moreover, the dependence of the generator on a time-delayed coefficient represents the main aspect of novelty arising in the analysis of the FBSDE system. Furthermore, we present an application for a large investor problem admitting a jump diffusion dynamic.
Throughout the article we mention some 
 possibilities to generalize the setting of our equations such as considering a dependence of $f$ also on a delayed term for the processes $Z$ and $U$ (see Remark \ref{delta} for more details) or in-depth analyzing the choice of a weaker topology (Remark \ref{weak}). Another possible modelling choice deals with considering a further delay term affecting the forward process, see \cite{ma} for more details. Furthermore, it might deserve attention to investigate a discretization scheme for this equation, recalling the work in \cite{be}, for the considered equations as to obtain a numerical algorithm based on Neural Networks methods to efficiently compute an approximated  solution for the considered FBSDE.

\begin{appendix}
	
	\section*{Proof of Theorem \ref{theorem1}}
	\label{AppA}

\begin{proof}
\label{proof}

The existence and the uniqueness are obtained by the Banach fixed point theorem. We consider $\phi$ fixed in $\Lambda$ and we define the map $\Gamma$ on $\mathcal{A}$ with $\mathcal{A} := \mathcal{C} \big( [0,T] \, ; \, \mathbb{S}_0^2 (\mathbb{R}) \big)$.

For $R \in \mathcal{A}$, we define $\Gamma (R): = Y$, where, for $t \in [0,T]$, the triple of adapted processes $\big( Y^t (s),Z^t(s),U^t (s,z) \big)_{s \in [t,T]}$ is the unique solution of the following BSDE

\begin{multline}
\label{BSDE iterative 1}
Y^t (s)  = h (X^{t, \phi}) +  \int_s^T F (r, X^{t, \phi}, Y^t (r), Z^t (r), \tilde{U}^{t} (r), R_r^t) dr   \\
-  \int_s^T  Z^t (r) dW(r)  -  \int_s^T  \int_{\mathbb{R}} U^{t} (r, z) \tilde{N} (dr, dz) \,  , \hspace{.3cm}  s \in [t,T]\ \, .
\end{multline}

For $s \in [0,t]$ we prolong the solution by taking $Y^t (s):= Y^s (s)$ and $Z^t (s) = U^{t,\phi} (s) := 0$.

{\em Step 1.} Let us first show that $\Gamma$ takes values in the Banach spaces $\mathcal{A}$. We take $R \in \mathcal{A}$ and we will prove that $Y  := \Gamma (R) \in \mathcal{A}$. Thus, for every $t \in [0,T]$ we have to show that
$$ Y^t \in  \, \mathbb{S}_0^2 (\mathbb{R}) \, ,$$
and that the application
$$[0,T] \ni t \mapsto Y^t \in \mathbb{S}_0^2 (\mathbb{R}) \, , $$
is continuous.

Let $t\in\left[ 0,T\right] $ be fixed and $t^{\prime}\in\lbrack0,T]$; with
no loss of generality, we will suppose that $t<t^{\prime}$ and $t^{\prime
}-t<\delta$.

Concerning the solution of the BSDE defined in \eqref{BSDE iterative 1}, we obtain the following estimate
\begin{equation*}
\begin{array}{l}
\mathbb{E}\big(\sup_{s\in \left[ 0,T\right] }|Y^{t}\left( s\right)
-Y^{t^{\prime }}\left( s\right) |^{2}\big)\smallskip \\ 
\leq \mathbb{E}\big(\sup_{s\in \left[ 0,t^{\prime }\right] }|Y^{t}\left(
s\right) -Y^{t^{\prime }}\left( s\right) |^{2}\big)+\mathbb{E}\big(%
\sup_{s\in \left[ t^{\prime },T\right] }|Y^{t}\left( s\right) -Y^{t^{\prime
}}\left( s\right) |^{2}\big)\smallskip \\ 
\leq 2\mathbb{E}\big(\sup_{s\in \left[ t,t^{\prime }\right] }|Y^{t}\left(
s\right) -Y^{t}\left( t\right) |^{2}\big)+2\mathbb{E}\big(\sup_{s\in \left[
t,t^{\prime }\right] }|Y^{t}\left( t\right) -Y^{s}\left( s\right) |^{2}\big)%
\smallskip \\ 
\quad +\mathbb{E}\big(\sup_{s\in \left[ t^{\prime },T\right] }|Y^{t}\left(
s\right) -Y^{t^{\prime }}\left( s\right) |^{2}\big).%
\end{array}%
\end{equation*}

We start by proving that
\begin{equation*}
\mathbb{E}\big[\sup_{s\in \left[ t,t^{\prime }\right] }|Y^{t}\left( s\right)
-Y^{t}\left( t\right) |^{2}\big] \rightarrow 0\,,
\end{equation*}%
as $t^{\prime }\rightarrow t$. By plugging the explicit solution and applying Doob's inequality, we get 
\begin{align*}
\mathbb{E}\big(\sup_{s\in \left[ t,t^{\prime }\right] } & |Y^{t}\left( s\right)
-Y^{t}\left( t\right) |^{2}\big)  = \mathbb{E}\Big[\sup_{s\in \left[ t,t^{\prime }\right] }  \Big|  \int_t^s F (r, X^{t, \phi}, Y^t (r), Z^t (r), \tilde{U}^{t} (r), R_r^t) dr  + \\
& -  \int_t^s  Z^t (r) dW(r)  -  \int_t^s  \int_{\mathbb{R}} U^{t} (r, z) \tilde{N} (dr, dz)  \Big|^2 \Big] \\
& \leq 3 \mathbb{E} \left[  \int_t^{t^\prime}  \left| F (r, X^{t, \phi}, Y^t (r), Z^t (r), \tilde{U}^{t} (r), R_r^t ) \right|^2 dr  \right] + \\
&  + 3 \mathbb{E} \left[ \sup_{s\in \left[ t,t^{\prime }\right] }  \left| \int_t^{s}  Z^t (r) dW(r) \right|^2 \right] + 3 \mathbb{E} \left[ \sup_{s\in \left[ t,t^{\prime }\right] }  \left|  \int_t^{s} \int_{\mathbb{R}} U^{t} (r, z) \tilde{N} (dr, dz)  \right|^2 \right] \\
& \leq 3 \mathbb{E} \left[  \int_t^{t^\prime}  \left| F (r, X^{t, \phi}, Y^t (r), Z^t (r), \tilde{U}^{t} (r), R_r^t) \right|^2 dr  \right] + \\
&  + 12 \mathbb{E} \left[   \int_t^{t^\prime}  \left| Z^t (r) \right|^2 dr  \right] + 12 \mathbb{E} \left[   \int_t^{t^\prime} \int_{\mathbb{R}} \left|  U^{t} (r, z)  \right|^2  \nu(dz) dr   \right] \, ,\\
\end{align*}

From the absolute continuity of the Lebesgue integral, we deduce that $$\mathbb{E}\big[\sup_{s\in \left[ t,t^{\prime }\right] }|Y^{t}\left( s\right) -Y^{t}\left( t\right) |^{2}\big] \rightarrow 0  \, ,$$
as $t^{\prime }\rightarrow t$.

Concerning the term $\mathbb{E}\big(\sup_{s\in \left[ t^{\prime },T\right]
}|Y^{t}\left( s\right) -Y^{t^{\prime }}\left( s\right) |^{2}\big)$ let us
denote for short, only throughout this step,%
\begin{equation*}
\begin{array}{c}
\Delta Y\left( r\right) :=Y^{t}\left( r\right) -Y^{t^{\prime }}\left(
r\right) ,\qquad \quad \qquad \Delta Z\left( r\right) :=Z^{t}\left( r\right) -Z^{t^{\prime
}}\left( r\right), \smallskip  \\ \Delta U\left( r,z\right) :=U^{t}\left( r,z\right) -U^{t^{\prime
}}\left( r,z\right), \qquad
\Delta R_r\left( r\right) :=R_r^{t}\left( r\right) -R_r^{t^{\prime }}\left(
r\right)%
\end{array}%
\end{equation*}%
and%
\begin{equation*}
\begin{array}{c}
\Delta h:=h(X^{t,\phi })-h(X^{t^{\prime },\phi }),\smallskip \\ 
\Delta F\left( r\right) :=F(r,X^{t,\phi },Y^{t}\left( r\right) ,Z^{t}\left(
r\right), \tilde{U}^{t}\left(r \right) ,R_{r}^{t}) -F(r,X^{t^{\prime },\phi },Y^{t}\left(
r\right) ,Z^{t}\left( r\right), \tilde{U}^{t}\left( r\right) ,R_{r}^{t}).%
\end{array}%
\end{equation*}%

We apply It\^{o}'s formula to $e^{\beta s}|\Delta Y\left( s\right) |^{2}$ and we derive, for any $\beta >0$ and any $s\in \left[
t^{\prime },T\right] ,$%
\begin{equation*}
\begin{array}{l}
\displaystyle e^{\beta s}|\Delta Y\left( s\right) |^{2}+\beta
\int_{s}^{T}e^{\beta r}|\Delta Y\left( r\right) |^{2}dr+\int_{s}^{T}e^{\beta
r}|\Delta Z\left( r\right) |^{2}dr + \int_{s}^{T} \int_{\mathbb{R}}e^{\beta
r}|\Delta U \left( r,z\right) |^{2} \nu(dz) dr \smallskip \\ 
\displaystyle=e^{\beta T}|\Delta Y\left( T\right)
|^{2}-2\int_{s}^{T}e^{\beta r} \Delta Y\left( r\right) \Delta
Z\left( r\right) \cdot dW\left( r\right)  -2\int_{s}^{T} \int_{\mathbb{R}} e^{\beta r} \Delta Y\left( r\right) \Delta
U\left( r, z\right)  \tilde{N} (dr,dz) \smallskip \\ 
\displaystyle +2\int_{s}^{T}e^{\beta r}  \Delta Y\left( r\right)
\left( F(r,X^{t,\phi }Y^{t}\left( r\right) ,Z^{t}\left( r\right)
, \tilde{U}^{t}\left( r \right), R_{r}^{t})  -F(r,X^{t^{\prime },\phi },Y^{t^{\prime }}\left(
r\right) ,Z^{t^{\prime }}\left( r\right) , \tilde{U}^{t^{\prime }}\left( r \right) , R_{r}^{t^{\prime
}}) \right) dr \smallskip  \\
\displaystyle\;\;  - \int_{s}^{T} \int_{\mathbb{R}} e^{\beta
r}|\Delta U \left( r,z\right) |^{2} \tilde{N} (dr,dz) \,. \smallskip \\%
\end{array}%
\end{equation*}%

We note that the following estimate
\begin{equation*}
\begin{array}{l}
\displaystyle\int_{s}^{T}e^{\beta r}\left(\int_{-\delta }^{0}\left( |\Delta
R\left( r+\theta \right) |^{2} \right)
\alpha (d\theta )\right)dr\smallskip =\int_{-\delta }^{0}\left[ \int_{s}^{T}e^{\beta r}\left(
|\Delta R\left( r+\theta \right) |^{2} \right)  dr\right]\alpha (d\theta )\smallskip \\ 
\displaystyle\leq e^{\beta \delta }\cdot \int_{-\delta }^{0}\alpha (d\theta
)\cdot \int_{0}^{T}e^{\beta r}\big(|\Delta R\left( r\right) |^{2} \big)dr \leq Te^{\beta \delta }\sup_{r\in \left[ 0,T\right] }\big(%
e^{\beta r}|\Delta R\left( r\right) |^{2}\big) \,,%
\end{array}%
\end{equation*}
holds. From assumptions $\mathrm{(A}_{3}\mathrm{)}$--$\mathrm{(A}_{5}\mathrm{)}$, we have for any $a>0,$%
\begin{equation*}
\begin{array}{l}
\displaystyle2\int_{s}^{T}e^{\beta r} \Delta Y\left( r\right)
\left( F(r,X^{t,\phi },Y^{t}\left( r\right) ,Z^{t}\left( r\right), \tilde{U}^{t}\left( r \right)
,R_{r}^{t}) 
-F(r,X^{t^{\prime },\phi },Y^{t^{\prime }}\left(
r\right) ,Z^{t^{\prime }}\left( r\right),     \tilde{U}^{t^{\prime }}\left( r \right) , R_{r}^{t^{\prime
}} \right) dr\smallskip \\ 
\displaystyle\leq a\int_{s}^{T}e^{\beta r}|\Delta Y\left( r\right) |^{2}dr+%
\frac{3}{a}\int_{s}^{T}e^{\beta r}|\Delta F\left( r\right) |^{2}dr \smallskip \\ 
\displaystyle \quad+ \frac{6L^{2}}{a}\int_{s}^{T}e^{\beta r}\left( |\Delta Y\left( r\right)
|^{2}+|\Delta Z\left( r\right) |^{2} + \int_{\mathbb{R}} |\Delta U \left( r,z \right) |^{2} \lambda(z) \nu(dz)  \right) dr\smallskip \\ 
\displaystyle\quad +\frac{3TKe^{\beta \delta }}{a}\sup_{r\in \left[ 0,T%
\right] }\big(e^{\beta r}|\Delta R\left( r\right) |^{2}\big) \, .
\end{array}%
\end{equation*}

Therefore we have 
\begin{equation*}
\begin{array}{l}
\displaystyle e^{\beta s}|\Delta Y\left( s\right) |^{2}+\left( \beta -a-%
\frac{6L^{2}}{a}\right) \int_{s}^{T}e^{\beta r}|\Delta Y\left( r\right)
|^{2}dr \smallskip \\
\displaystyle +\left( 1-\frac{6L^{2}}{a}\right) \int_{s}^{T}e^{\beta r}|\Delta
Z\left( r\right) |^{2}dr +\left( 1-\frac{6L^{2}}{a}\right) \int_{s}^{T} \int_{\mathbb{R}} e^{\beta r}|\Delta
U\left( r,z\right) |^{2} \lambda(z) \nu (dz) dr \smallskip \\ 
\displaystyle\leq e^{\beta T}|\Delta Y\left( T\right) |^{2}+\frac{3}{a}%
\int_{s}^{T}e^{\beta r}|\Delta F\left( r\right)
|^{2}dr \smallskip \\
\displaystyle \quad -2\int_{s}^{T}e^{\beta r} \Delta Y\left( r\right)  \Delta
Z\left( r\right) \cdot dW\left( r\right) -2\int_{s}^{T} \int_{\mathbb{R}} e^{\beta r}  \Delta Y\left( r\right) \Delta
U\left( r, z\right)  \tilde{N} (dr,dz)  \smallskip \\ 
\displaystyle\quad +\frac{3TKe^{\beta \delta }}{a}\sup_{r\in \left[ 0,T%
\right] }e^{\beta r}|\Delta R\left( r\right) |^{2}  \, .
\end{array}%
\end{equation*}%
We now choose $\beta ,a>0$ such that%
\begin{equation}
a+\frac{6L^{2}}{a}<\beta \quad \text{and}\quad \frac{6L^{2}}{a}<1,
\label{restriction_parameters}
\end{equation}%
hence we obtain%
\begin{equation}
\begin{array}{l}
\displaystyle\left( 1-\frac{6L^{2}}{a}\right) \mathbb{E} \Bigg[ \int_{s}^{T}e^{\beta
r}|\Delta Z\left( r\right) |^{2}dr +  \int_{s}^{T} \int_{\mathbb{R}}e^{\beta r}|\Delta U \left( r,z\right) |^{2} \lambda(z) \nu (dz)  dr \Bigg] \leq \smallskip \\ 
\displaystyle \quad  \mathbb{E}\big(e^{\beta T}|\Delta
h|^{2}\big)+\frac{3}{a}\mathbb{E}\int_{s}^{T}e^{\beta r}|\Delta F\left(
r\right) |^{2}dr  +\frac{3TKe^{\beta \delta }}{a}\mathbb{E}\big(\sup_{r\in %
\left[ 0,T\right] }e^{\beta r}|\Delta R\left( r\right) |^{2}\big) .%
\end{array}
\label{technical ineq 14}
\end{equation}%
By Burkholder--Davis--Gundy's inequality, we have%
\begin{equation*}
\begin{array}{l}
\displaystyle2\mathbb{E}\left[ \sup_{s\in \left[ t^{\prime },T\right] }\Bigg|%
\int_{s}^{T}e^{\beta r} \Delta Y\left( r\right) \Delta Z\left(
r\right) \cdot dW\left( r\right) \Bigg|\right] \smallskip \\ 
\displaystyle\leq \frac{1}{4}\mathbb{E}\left[\sup_{s\in \left[ t^{\prime },T%
\right] }e^{\beta s}|\Delta Y\left( s\right) |^{2}\right]+144\mathbb{E}%
\int_{t^{\prime }}^{T}e^{\beta r}|\Delta Z\left( r\right) |^{2}dr.%
\end{array}%
\end{equation*}%
and 
\begin{equation*}
\begin{array}{l}
\displaystyle2\mathbb{E}\left[\sup_{s\in \left[ t^{\prime },T\right] }\Bigg|%
\int_{s}^{T} \int_{\mathbb{R}} e^{\beta r} \Delta Y\left( r\right) \Delta U\left(
r\right)  \tilde{N} (dr,dz) \Bigg|\right]\smallskip \\ 
\displaystyle\leq \frac{1}{4}\mathbb{E}\left[\sup_{s\in \left[ t^{\prime },T%
\right] }e^{\beta s}|\Delta Y\left( s\right) |^{2}\right]+144\mathbb{E}%
\int_{t^{\prime }}^{T} \int_{\mathbb{R} }e^{\beta r}|\Delta U\left( r,z\right) |^{2} \lambda(z) \nu (dz) dr.%
\end{array}%
\end{equation*}%
which immediately implies 
\begin{equation*}
\begin{array}{l}
\displaystyle\frac{1}{2}\mathbb{E}\left[ \sup_{s\in \left[ t^{\prime },T\right]
}e^{\beta s}|\Delta Y\left( s\right) |^{2}\right]\leq \mathbb{E}\big(e^{\beta
T}|\Delta h|^{2}\big)+\frac{3}{a}\mathbb{E}\int_{t^{\prime }}^{T}e^{\beta
r}|\Delta F\left( r\right) |^{2}dr\smallskip \\ 
\displaystyle\quad +\frac{3TKe^{\beta \delta }}{a}\mathbb{E}\big(\sup_{r\in %
\left[ 0,T\right] }e^{\beta r}|\Delta R\left( r\right) |^{2}\big) + 144 \, \mathbb{E}\int_{t^{\prime }}^{T}|\Delta Z\left( r\right)
|^{2}dr + \smallskip \\
\displaystyle  \quad + 144 \, \mathbb{E}%
\int_{t^{\prime }}^{T} \int_{\mathbb{R} }e^{\beta r}|\Delta U\left( r,z\right) |^{2} \lambda(z) \nu (dz) dr \,.%
\end{array}%
\end{equation*}%
Hence, we have%
\begin{equation}
\begin{array}{l}
\displaystyle\frac{1}{2}\mathbb{E}\left(\sup_{s\in \left[ t^{\prime },T\right]
}e^{\beta s}|\Delta Y\left( s\right) |^{2}\right)\leq \mathbb{E}\big(e^{\beta
T}|\Delta h|^{2}\big)+\frac{3}{a}C_{1}\mathbb{E}\int_{t^{\prime
}}^{T}e^{\beta r}|\Delta F\left( r\right) |^{2}dr\smallskip \\ 
\displaystyle\quad +\frac{3TKe^{\beta \delta }C_{1}}{a}\,\mathbb{E}\big(%
\sup_{r\in \left[ 0,T\right] }e^{\beta r}|\Delta R\left( r\right) |^{2}\big) \, ,
\smallskip 
\end{array}
\label{technical ineq 15}
\end{equation}%
where%
\begin{equation*}
C_{1}:=1+\frac{144}{1-6L^{2}/a}\,.
\end{equation*}%

Exploiting thus assumptions $\mathrm{(A}_{3}\mathrm{)}$ and $\mathrm{(A}_{5}%
\mathrm{)}$ together with the fact that $X^{\cdot ,\phi }$ is continuous and
bounded, we have%
\begin{equation*}
C_{1}\mathbb{E}\big(e^{\beta T}|\Delta h|^{2}\big)+\frac{3}{a}C_{1}\mathbb{E}%
\int_{t^{\prime }}^{T}e^{\beta r}|\Delta F\left( r\right) |^{2}dr\rightarrow
0\quad \text{as }t^{\prime }\rightarrow t \, .
\end{equation*}%
Since $ R  \in \mathcal{A} $, and therefore
we have 
\begin{equation*}
\mathbb{E}\left[ \sup_{r\in \left[ 0,T\right] }e^{\beta r}|\Delta R\left(
r\right) |^{2}\right]\,\rightarrow 0 \, , 
\end{equation*}%
as $t^{\prime }\rightarrow t$, we have%
\begin{equation}
\mathbb{E}\left[\sup_{s\in \left[ t^{\prime },T\right] }e^{\beta s}|\Delta
Y\left( s\right) |^{2}\right]\rightarrow 0 \, ,  \label{technical ineq 16}
\end{equation}
as $t^{\prime }\rightarrow t$. 

We are left to show that the term $\mathbb{E}\big(\sup_{s\in\left[
t,t^{\prime}\right] }|Y^{t}\left( t\right) -Y^{s}\left( s\right) |^{2}\big)$
is also converging to $0$ as $t^{\prime}\rightarrow t$.

Since the map $t\mapsto Y^{t}\left( t\right) $ is deterministic, we have
from equation (\ref{BSDE iterative 1}),%
\begin{equation*}
\begin{array}{l}
\displaystyle Y^{t}\left( t\right) -Y^{s}\left( s\right) =\mathbb{E}\big[%
Y^{t}\left( t\right) -Y^{s}\left( s\right) \big]\smallskip \\ 
\displaystyle=\mathbb{E}\big[h(X^{t,\phi })-h(X^{s,\phi })\big]+\mathbb{E}%
\int_{t}^{T}F(r,X^{t,\phi },Y^{t}\left( r\right) ,Z^{t}\left( r\right), \tilde{U}^{t}\left( r\right), R_{r}^{t})dr\smallskip \\ 
\displaystyle\quad -\mathbb{E}\int_{s}^{T}F(r,X^{s,\phi },Y^{s}\left(
r\right) ,Z^{s}\left( r\right) , \tilde{U}^{s}\left( r\right) ,R_{r}^{s})dr\smallskip \\ 
\displaystyle=\mathbb{E}\big[h(X^{t,\phi })-h(X^{s,\phi })\big]+\mathbb{E}%
\int_{t}^{s}F(r,X^{t,\phi },Y^{t}\left( r\right) ,Z^{t}\left( r\right)
,   \tilde{U}^{t}\left( r\right), R_{r}^{t})dr\smallskip \\ 
\displaystyle\quad +\mathbb{E}\int_{s}^{T}\big[F(r,X^{t,\phi },Y^{t}\left(
r\right) ,Z^{t}\left( r\right), \tilde{U}^{t}\left( r\right),R_{r}^{t}) \smallskip \\ 
\displaystyle\qquad -F(r,X^{s,\phi
},Y^{s}\left( r\right) ,Z^{s}\left( r\right) , \tilde{U}^{s}\left( r\right) , R_{r}^{s})\big]dr.%
\end{array}%
\end{equation*}%
Using then assumption \textrm{(A}$_{3}$\textrm{)} we have%
\begin{equation*}
\begin{array}{l}
\displaystyle|Y^{t}\left( t\right) -Y^{s}\left( s\right) |\leq \mathbb{E}%
\big|h(X^{t,\phi })-h(X^{s,\phi })\big|+\mathbb{E}\int_{t}^{s}L\left(%
|Y^{t}\left( r\right) |+|Z^{t}\left( r\right) | + \Big| \int_{\mathbb{R}} U^{t}\left( r,z\right) \lambda(z) \nu(dz) \Big| \right)dr\smallskip \\ 
\displaystyle+\sqrt{K\int_{t}^{s}\mathbb{E}\left[ \int_{-\delta }^{0}\left(
|R^{t}\left( r+\theta \right) |^{2} \right) \alpha (d\theta )\right] dr}\cdot \sqrt{s-t}  +\mathbb{E}\int_{t}^{s}\big|F(r,X^{t,\phi },0,0,0,0,0,0)\big|%
dr\smallskip \\ 
\displaystyle+\mathbb{E}\int_{s}^{T}\big|F(r,X^{t,\phi },Y^{t}\left(
r\right) ,Z^{t}\left( r\right) , \tilde{U}^{t}\left( r\right), R_{r}^{t})-F(r,X^{s,\phi
},Y^{t}\left( r\right) ,Z^{t}\left( r\right) , \tilde{U}^{t}\left( r\right) ,R_{r}^{t})\big|%
dr\smallskip \\ 
\displaystyle+\mathbb{E}\int_{s}^{T}L\left(|Y^{t}\left( r\right) -Y^{s}\left(
r\right) |+|Z^{t}\left( r\right) -Z^{s}\left( r\right) | + \Big| \int_{\mathbb{R}}U^{t}\left( r,z\right) -U^{s}\left( r,z\right) \lambda (z) \nu(dz) \Big|\right)dr\smallskip \\ 
\displaystyle+\sqrt{K(T-s)\int_{s}^{T}\mathbb{E}\left[\int_{-\delta
}^{0}\left( |R^{t}\left( r+\theta \right) -R^{s}\left( r+\theta \right)
|^{2} \Big|\right) \alpha (d\theta )\right]dr} \, ,
\end{array}%
\end{equation*}%
and therefore we obtain%
\begin{equation*}
\begin{array}{l}
\displaystyle|Y^{t}\left( t\right) -Y^{s}\left( s\right) | \leq \mathbb{E}\big|h(X^{t,\phi })-h(X^{s,\phi })\big| \smallskip \\ 
\displaystyle+L\sqrt{%
s-t}\sqrt{T\mathbb{E}\sup_{r\in \left[ 0,T\right] }|Y^{t}\left( r\right)|^{2}+\mathbb{E}\int_{0}^{T}|Z^{t}\left( r\right) |^{2}dr +\mathbb{E}\int_{0}^{T} \int_\mathbb{R} |U^{t}\left( r,z \right) |^{2} \lambda(z) \nu (dz) dr}\smallskip \\ 
\displaystyle\quad +\sqrt{K}\sqrt{s-t}\sqrt{T\mathbb{E}\sup_{r\in \left[ 0,T%
\right] }|R^{t}\left( r\right) |^{2} }  +\left( s-t\right) M(1+\mathbb{E}||X^{t,\phi
}||_{T}^{p})\smallskip \\ 
\displaystyle\quad +\mathbb{E}\int_{s}^{T}\big|F(r,X^{t,\phi },Y^{t}\left(
r\right) ,Z^{t}\left( r\right), \tilde{U}^{t}\left( r\right) ,R_{r}^{t})-F(r,X^{s,\phi
},Z^{t}\left( r\right), \tilde{U}^{t}\left( r\right) ,R_{r}^{t})\big|%
dr\smallskip \\ 
\displaystyle +L\sqrt{T-s}\sqrt{T\mathbb{E}\sup_{r\in \left[ s,T\right]
}|Y^{t}\left( r\right) -Y^{s}\left( r\right) |^{2}+\mathbb{E}%
\int_{s}^{T}|Z^{t}\left( r\right) -Z^{s}\left( r\right) |^{2}dr +\mathbb{E}%
\int_{s}^{T}| \int_{\mathbb{R}} U^{t}\left( r,z\right) -U^{s,z}\left( r,z\right) |^{2} \lambda(z) \nu (dz) dr}\smallskip \\ 
\displaystyle +\sqrt{K}\sqrt{T-s}\sqrt{T\mathbb{E}\sup_{r\in \left[ 0,T%
\right] }|R^{t}\left( r\right) -R^{s}\left( r\right) |^{2}}\,.%
\end{array}%
\end{equation*}%
Taking again into account the fact that $ R  \in \mathcal{A}$, previous step and assumptions $\mathrm{(A}_{3}\mathrm{)}
$ and $\mathrm{(A}_{5}\mathrm{)}$, we infer that 
\begin{equation}
\mathbb{E}\left[\sup_{s\in \left[ t,t^{\prime }\right] }|Y^{t}\left( t\right)
-Y^{s}\left( s\right) |\right]\rightarrow 0,\quad \text{as }t^{\prime
}\rightarrow t.  \label{technical ineq 17}
\end{equation}

Concerning the term $\mathbb{E}\int_{0}^{T}|Z^{t}\left( r\right)
-Z^{t^{\prime}}\left( r\right) |^{2}dr$, we see that%
\begin{equation*}
\begin{array}{l}
\displaystyle\mathbb{E}\int_{0}^{T}|Z^{t}\left( r\right)
-Z^{t^{\prime}}\left( r\right) |^{2}dr=\mathbb{E}\int_{0}^{t^{\prime}}|Z^{t}%
\left( r\right) -Z^{t^{\prime}}\left( r\right) |^{2}dr+\mathbb{E}%
\int_{t^{\prime}}^{T}|Z^{t}\left( r\right) -Z^{t^{\prime}}\left( r\right)
|^{2}dr\smallskip \\ 
\displaystyle=\mathbb{E}\int_{t}^{t^{\prime}}|Z^{t}\left( r\right) |^{2}dr+%
\mathbb{E}\int_{t^{\prime}}^{T}|Z^{t}\left( r\right) -Z^{t^{\prime}}\left(
r\right) |^{2}dr\,,%
\end{array}%
\end{equation*}
hence, by (\ref{technical ineq 14}), 
\begin{equation}
\mathbb{E}\int_{0}^{T}|Z^{t}\left( r\right) -Z^{t^{\prime}}\left( r\right)
|^{2}dr\rightarrow0,\quad\text{as }t^{\prime}\rightarrow t.
\label{technical ineq 18}
\end{equation}
Analogously, we can infer that $$\mathbb{E}\int_{0}^{T} \int_{\mathbb{R}}|U^{t}\left( r,z\right)
-U^{t^{\prime}}\left( r,z\right) |^{2} \lambda (z) \nu (dz) dr \rightarrow 0 \, , $$
as $ t^{\prime} \rightarrow t$. 

\medskip

\noindent\textbf{\textit{Step II.}}

{\em Step 2.} We are going to prove that $\Gamma$ is a contraction on $\mathcal{A}$  with respect to the norm%

\begin{equation*}
\displaystyle||Y||_{\mathcal{A}} := \left( \sup_{t\in\left[ 0,T\right] }\mathbb{E}\left[%
\sup_{r\in\left[ 0,T\right] }e^{\beta r}|Y^{t}\left( r\right) |^{2}\right] \right)^{1/2}%
,\smallskip
\end{equation*}

Let us recall that $\Gamma:\mathcal{A}\rightarrow \mathcal{A%
}$ is defined by $\Gamma\left( R \right) =
Y  $ being $Y$ the process coming from the solution of the BSDE (\ref%
{BSDE iterative 1}).

Let us consider $ R^{1},  R^{2} \in%
\mathcal{A}$ and $ Y^{1} :=\Gamma\left(
R^{1}\right) $, $Y^{2} :=\Gamma\left(
R^{2}\right) $. For the sake of brevity, we will denote in what follows%
\begin{equation*}
\begin{array}{c}
\Delta F^{t}\left( r\right) :=F(r,X^{t,\phi},Y^{1,t}\left( r\right)
,Z^{1,t}\left( r\right), \tilde{U}^{1,t}\left( r\right) ,R_{r}^{1,t})- F(r,X^{t,\phi},Y^{2,t}\left( r\right) ,Z^{2,t}\left(
r\right), \tilde{U}^{2,t}\left(
r\right) ,R_{r}^{2,t}),\smallskip \\ 
\Delta R^{t}\left( r\right) :=R^{1,t}\left( r\right) -R^{2,t}\left( r\right), \quad \quad
\Delta Y^{t}\left( r\right) :=Y^{1,t}\left( r\right) -Y^{2,t}\left( r\right)
\smallskip \\  \Delta Z^{t}\left( r\right) :=Z^{1,t}\left( r\right) -Z^{2,t}\left(
r\right) , \quad \quad\Delta U^{t}\left( r\right) :=U^{1,t}\left( r\right) -U^{2,t}\left(
r\right) .%
\end{array}%
\end{equation*}

Proceeding as in \textbf{\textit{Step I}}, we have from It\^{o}'s formula,
for any $s\in \left[ t,T\right] $ and $\beta >0$,%
\begin{equation}
\begin{array}{l}
\displaystyle e^{\beta s}|\Delta Y^{t}\left( s\right) |^{2}+\beta
\int_{s}^{T}e^{\beta r}|\Delta Y^{t}\left( r\right)
|^{2}dr+\int_{s}^{T}e^{\beta r}|\Delta Z^{t}\left( r\right) |^{2}dr  + \int_{s}^{T} \int_{\mathbb{R}}e^{\beta
r}|\Delta U^t \left( r,z\right) |^{2} \lambda(z) \nu(dz) dr \smallskip 
\\ 
\displaystyle=2\int_{s}^{T}e^{\beta r} \Delta Y^{t}\left( r\right)
\Delta F^{t}\left( r\right)  dr-2\int_{s}^{T}e^{\beta r}
\Delta Y^{t}\left( r\right) \Delta Z^{t}\left( r\right) \cdot  dW\left(
r\right) \smallskip \\
\displaystyle \quad -2\int_{s}^{T} \int_{\mathbb{R}} e^{\beta r} \Delta Y^t \left( r\right) \Delta
U^t \left( r, z\right)  \tilde{N} (dr,dz)     \,. %
\end{array}
\label{technical ineq 1}
\end{equation}

Noticing that it holds 
\begin{equation*}
\begin{array}{l}
\displaystyle\frac{2K}{a}\int_{s}^{T}e^{\beta r}\left(\int_{-\delta}^{0}\left(%
\left\vert \Delta R^{t}(r+\theta)\right\vert ^{2} \right)\alpha(d\theta)\right)dr \leq\frac{2K}{a}\int_{-\delta}^{0}\left(\int_{s}^{T}e^{\beta r}%
\left(\left\vert \Delta R^{t}(r+\theta)\right\vert ^{2} \right)dr \right)\alpha(d\theta)\smallskip \\ 
\displaystyle\leq\frac{2K}{a}\int_{-\delta}^{0}\left(\int_{s+r}^{T+r}e^{\beta%
\left( r^{\prime}-\theta\right) }\left(\left\vert \Delta
R^{t}(r^{\prime})\right\vert ^{2} \right)dr^{\prime}\right)\alpha(d\theta) \leq \frac{2K}{a}\int_{-\delta}^{0}e^{-\beta\theta}\alpha
(d\theta)\cdot\int_{s-\delta}^{T}e^{\beta r}\left(\left\vert \Delta
R^{t}(r)\right\vert ^{2} \Big \vert\right)%
dr\smallskip \\ 
\displaystyle\leq\frac{2Ke^{\beta\delta}}{a}\int_{s-\delta}^{T}e^{\beta r}%
\left(\left\vert \Delta R^{t}(r)\right\vert ^{2} \right)dr.%
\end{array}%
\end{equation*}
we immediately have, from assumptions $\mathrm{(A}_{4}\mathrm{)}$--$\mathrm{%
(A}_{6}\mathrm{)}$, that for any $a>0$, 
\begin{equation}
\begin{array}{l}
\displaystyle2\Bigg|\int_{s}^{T}e^{\beta r} \Delta Y^{t}\left( r\right)
\Delta F^{t}\left( r\right)  dr\Bigg|\leq2\int_{s}^{T}e^{\beta
r}| \Delta Y^{t}\left( r\right) ,\Delta F^{t}\left( r\right)
|dr\smallskip \\ 
\displaystyle\leq a\int_{s}^{T}e^{\beta r}|\Delta Y^{t}\left( r\right) |^{2}+%
\frac{1}{a}\int_{s}^{T}e^{\beta r}|\Delta F^{t}\left( r\right)
|^{2}dr\smallskip \\ 
\displaystyle\leq a\int_{s}^{T}e^{\beta r}|\Delta Y^{t}\left( r\right) |^{2}+%
\frac{2}{a}\int_{s}^{T}e^{\beta r}L^{2}\left( |\Delta Y^{t}\left( r\right)
|+|\Delta Z^{t}\left( r\right) + \Big\vert \int_{\mathbb{R}} \Delta U^{t} \left( r ,z\right) \lambda (z) \nu(dz)  \Big\vert \right) ^{2}dr\smallskip \\ 
\displaystyle\quad+\frac{2}{a}\int_{s}^{T}e^{\beta r}\left(K\int_{-\delta}^{0}%
\left(\left\vert \Delta R^{t}(r+\theta)\right\vert ^{2} \right)\alpha(d\theta)\right)dr\smallskip \\ 
\displaystyle\leq a\int_{s}^{T}e^{\beta r}|\Delta Y^{t}\left( r\right) |^{2}+%
\frac{4L^{2}}{a}\int_{s}^{T}e^{\beta r}\left(|\Delta Y^{t}\left( r\right)
|^{2}+|\Delta Z^{t}\left( r\right) |^{2} + \Big\vert \int_{\mathbb{R}} \Delta U^{t} \left( r ,z\right) \lambda (z) \nu(dz)  \Big\vert^2 \right)dr\smallskip \\ 
\displaystyle\quad+\frac{2Ke^{\beta\delta}}{a}\int_{s-\delta}^{T}e^{\beta r}%
\left(\left\vert \Delta R^{t}(r)\right\vert ^{2}  \right)dr\,.%
\end{array}
\label{technical ineq 2}
\end{equation}

Therefore equation (\ref{technical ineq 1}) yields%
\begin{equation}
\begin{array}{l}
\displaystyle e^{\beta s}|\Delta Y^{t}\left( s\right) |^{2}+\left( \beta -a-%
\frac{4L^{2}}{a}\right) \int_{s}^{T}e^{\beta r}|\Delta Y^{t}\left( r\right)
|^{2}dr+\left( 1-\frac{4L^{2}}{a}\right) \int_{s}^{T}e^{\beta r}|\Delta
Z^{t}\left( r\right) |^{2}dr \smallskip \\ 
\displaystyle \quad + \left( 1-\frac{4L^{2}}{a}\right)  \int_{s}^{T} \int_{\mathbb{R}_0}e^{\beta
r}|\Delta U^t \left( r,z\right) |^{2} \lambda(z) \nu(dz) dr \smallskip 
\\ 
\displaystyle \leq \frac{2Ke^{\beta \delta }}{a}T\sup_{r\in \left[ 0,T\right]
}e^{\beta r}|\Delta R^{t}\left( r\right) |^{2} -2\int_{s}^{T}e^{\beta r} \Delta Y^{t}\left(
r\right) \Delta Z^{t}\left( r\right) \cdot  dW\left( r\right)  \smallskip \\ 
\displaystyle\quad -2\int_{s}^{T} \int_{\mathbb{R}_0} e^{\beta r} \Delta Y^t \left( r\right) ,\Delta
U^t \left( r, z\right)  \tilde{N} (dr,dz)   .%
\end{array}
\label{technical ineq 19}
\end{equation}%

Let now $\beta ,a>0$ satisfying%
\begin{equation}
\beta >a+\frac{4L^{2}}{a}\quad \text{and}\quad 1>\frac{4L^{2}}{a}\,,
\label{restriction_parameters2}
\end{equation}%
we have%
\begin{equation}
\begin{array}{l}
\displaystyle\left( 1-\frac{4L^{2}}{a}\right) \mathbb{E}\int_{s}^{T}e^{\beta
r}|\Delta Z^{t}\left( r\right) |^{2}dr  + \left( 1-\frac{4L^{2}}{a}\right)  \mathbb{E} \int_s^T \int_{\mathbb{R}_0} e^{\beta r} | \Delta U^{t} (r, z)|^2  \lambda (z) \nu(dz) dr  \smallskip \\ 
\displaystyle \quad
\leq \frac{2TKe^{\beta \delta }}{a}\mathbb{E}\left(\sup_{r\in %
\left[ 0,T\right] }e^{\beta r}|\Delta R^{t}\left( r\right) |^{2}\right) 
\end{array}
\label{technical ineq 3}
\end{equation}

%
%

Exploiting now Burkholder--Davis--Gundy's inequality, we have%
\begin{equation*}
\begin{array}{l}
\displaystyle2\mathbb{E}\left[\sup_{s\in \left[ t,T\right] }\Big|%
\int_{s}^{T}e^{\beta r} \Delta Y^{t}\left( r\right) \Delta
Z^{t}\left( r\right) \cdot dW\left( r\right) \Big|\right]\smallskip \\ 
\displaystyle\leq \frac{1}{4}\mathbb{E}\left[ \sup_{s\in \left[ t,T\right]
}e^{\beta s}|\Delta Y^{t}\left( s\right) |^{2}\right]+144\mathbb{E}%
\int_{t}^{T}e^{\beta r}|\Delta Z^{t}\left( r\right) |^{2}dr\,,%
\end{array}%
\end{equation*}%

and, analogously, 

\begin{equation*}
\begin{array}{l}
\displaystyle2\mathbb{E}\left[\sup_{s\in \left[ t^{\prime },T\right] }\Big|%
\int_{s}^{T} \int_{\mathbb{R}} e^{\beta r} \Delta Y^t \left( r\right) \Delta U^t \left(
r\right)  \tilde{N} (dr,dz) \Big|\right] \smallskip \\ 
\displaystyle\leq \frac{1}{4}\mathbb{E}\left[\sup_{s\in \left[ t,T%
\right] }e^{\beta s}|\Delta Y^t \left( s\right) |^{2}\right]+144 \mathbb{E}%
\int_{t}^{T} \int_{\mathbb{R}\setminus \{0\} }e^{\beta r}|\Delta U^t \left( r,z\right) |^{2} \lambda(z) \nu (dz) dr.%
\end{array}%
\end{equation*}%
which implies 
\begin{equation*}
\begin{array}{l}
\displaystyle\mathbb{E}\left[\sup_{s\in \left[ t,T\right] }e^{\beta s}|\Delta
Y^{t}\left( s\right) |^{2}\right] \leq \frac{2Ke^{\beta \delta }}{a}T\mathbb{E}%
\big(\sup_{s\in \left[ 0,T\right] }e^{\beta s}|\Delta R^{t}\left( s\right)
|^{2}\big) + \smallskip \\ 
\displaystyle\quad  +2\mathbb{E}\left[\sup_{s\in \left[ t,T\right] }\Big|%
\int_{s}^{T}e^{\beta r} \Delta Y^{t}\left( r\right) \Delta
Z^{t}\left( r\right) \cdot dW\left( r\right) \Big|\right]  +  \smallskip \\ 
\displaystyle\quad + 2\mathbb{E}\left[\sup_{s\in \left[ t^{\prime },T\right] }\Big|%
\int_{s}^{T} \int_{\mathbb{R}} e^{\beta r} \Delta Y^t \left( r\right) \Delta U^t \left(
r\right)  \tilde{N} (dr,dz) \Big|\right]  \smallskip \\ 
\displaystyle \leq \frac{2Ke^{\beta \delta }}{a}T\mathbb{E}\left[\sup_{s\in %
\left[ 0,T\right] }e^{\beta s}|\Delta R^{t}\left( s\right) |^{2}\right] +
\smallskip \\ 
\displaystyle \quad +\frac{1}{4}\mathbb{E}\left[\sup_{s\in \left[ t,T\right]
}e^{\beta s}|\Delta Y^{t}\left( s\right) |^{2}\right]+144\mathbb{E}%
\int_{t}^{T}e^{\beta r}|\Delta Z^{t}\left( r\right) |^{2}dr \smallskip \\ 
\displaystyle \quad 
+ \frac{1}{4}\mathbb{E}\left[\sup_{s\in \left[ t^,T%
\right] }e^{\beta s}|\Delta Y^t \left( s\right) |^{2}\right]+144 \mathbb{E}%
\int_{t}^{T} \int_{\mathbb{R}_0}e^{\beta r}|\Delta U^t \left( r,z\right) |^{2} \lambda(z) \nu (dz) dr
\,.%
\end{array}%
\end{equation*}%
Hence, we have%
\begin{equation}
\begin{array}{l}
\displaystyle\mathbb{E}\left[\sup_{s\in \left[ t,T\right] }e^{\beta s}|\Delta
Y^{t}\left( s\right) |^{2}\right] \leq \frac{4TKe^{\beta \delta }}{a}C_{1}\mathbb{E}\left[%
\sup_{s\in \left[ 0,T\right] }e^{\beta s}|\Delta R^{t}\left( s\right] |^{2}%
\right]
\end{array}
\label{technical ineq 5}
\end{equation}%
where we have denoted by $C_{1}:=1+\frac{144}{1-4L^{2}/a}\,.\smallskip $

Let us now consider the term $\mathbb{E}\big(\sup_{s\in \left[ 0,t\right]
}e^{\beta s}|\Delta Y\left( s\right) |^{2}\big)$. From equation (\ref{BSDE
iterative 1}), we see that, 
\begin{equation}
\begin{array}{l}
\displaystyle\mathbb{E}\big(\sup_{s\in \left[ 0,t\right] }e^{\beta s}|\Delta
Y^{t}\left( s\right) |^{2}\big)=\mathbb{E}\big(\sup_{s\in \left[ 0,t\right]
}e^{\beta s}|Y^{1,t}(s)-Y^{2,t}(s)|^{2}\big)\smallskip \\ 
\displaystyle=\mathbb{E}\big(\sup_{s\in \left[ 0,t\right] }e^{\beta
s}|Y^{1,s}(s)-Y^{2,s}(s)|^{2}\big)=\sup_{s\in \left[ 0,t\right] }e^{\beta
s}|\Delta Y^{s}\left( s\right) |^{2}=\sup_{s\in \left[ 0,t\right] }\mathbb{E}%
\big(e^{\beta s}|\Delta Y^{s}\left( s\right) |^{2}\big)%
\end{array}
\label{technical ineq 6}
\end{equation}%
so that, exploiting It\^{o}'s formula and proceeding as above, we obtain%
\begin{equation}
\displaystyle\mathbb{E}\big(e^{\beta s}|\Delta Y^{s}\left( s\right) |^{2}%
\big)\leq \frac{2TKe^{\beta \delta }}{a}\mathbb{E}\left[ \sup_{r\in \left[ 0,T%
\right] }e^{\beta r}|\Delta R^{s}\left( r\right) |^{2}\right]   \, .  \label{technical ineq 6'}
\end{equation}%
Thus from inequalities (\ref{technical ineq 3}--\ref{technical ineq 6'}) we
obtain%
\begin{equation*}
\begin{array}{l}
\displaystyle\mathbb{E}\left[\sup_{s\in \left[ 0,T\right] }e^{\beta r}|\Delta
Y^{t}\left( s\right) |^{2}\right]+\mathbb{E}\int_{0}^{T}e^{\beta r}|\Delta
Z^{t}\left( r\right) |^{2}dr + \mathbb{E} \int_{0}^{T} \int_{\mathbb{R}}e^{\beta
r}|\Delta U^t \left( r,z\right) |^{2} \lambda(z) \nu(dz) dr \smallskip \\ 
\displaystyle\leq \frac{4TKe^{\beta \delta }}{a}C_{1}\mathbb{E}\left[%
\sup_{s\in \left[ 0,T\right] }e^{\beta s}|\Delta R^{t}\left( s\right) |^{2}%
\right]  +\frac{2TKe^{\beta \delta }}{a\left( 1-4L^{2}/a\right) }%
\mathbb{E}\left[ \sup_{r\in \left[ 0,T\right] }e^{\beta r}|\Delta R^{t}\left(
r\right) |^{2}\right] \smallskip
\\ 

\displaystyle\quad +\frac{2TKe^{\beta \delta }}{a} \sup_{s\in \left[ 0,t%
\right] }\mathbb{E}\left[\sup_{r\in \left[ 0,T\right] }e^{\beta r}|\Delta
R^{s}\left( r\right) |^{2}\right] \, .
\end{array}%
\end{equation*}%
Them, passing to the supremum for $t\in \left[ 0,T\right] $ we get%
\begin{equation*}
\begin{array}{l}
||Y^{1}-Y^{2}||_{\mathcal{A}}^{2}  \leq \frac{%
2Ke^{\beta \delta }}{a}\Big(3+\frac{289}{1-4L^{2}/a}\Big)\max \left\{
1,T\right\} \Big[||R^{1}-R^{2}||_{\mathcal{A}}^{2}\Big].
\end{array}%
\end{equation*}%
By choosing now $a:=\frac{4L^{2}}{\chi }$ and $\beta $ slightly bigger
than $\chi +\frac{4L^{2}}{\chi }$, condition (\ref%
{restriction_parameters2}) is satisfied and, by restriction $\mathrm{(C)}$
we have%
\begin{equation}
\frac{2Ke^{\beta \delta }}{a}\Big(3+\frac{289}{1-4L^{2}/a}\Big)\max \left\{
1,T\right\} <1\,.  \label{restriction 1}
\end{equation}

Eventually, since $R$ is chosen arbitrarily, it follows that the
application $\Gamma$ is a contraction on $\mathcal{A}$. Therefore, there exists a unique fixed point $\Gamma
(R)=Y \in\mathcal{A}$ and this finishes the proof of
the existence and a uniqueness of a solution to BSDE with delay and driven by Lèvy process, described by Eq. \eqref{backward}.\hfill

\end{proof}
\end{appendix}

\end{document}